\documentclass[12pt,a4paper, oneside,reqno]{amsart}
\usepackage[colorlinks]{hyperref}
\usepackage{amsthm,array,mathrsfs}
\usepackage{amsmath,amssymb,enumitem}

\usepackage{tikz}
\usetikzlibrary{patterns}
\usetikzlibrary{calc}

\newtheorem{theorem}{Theorem}[section]
\newtheorem{proposition}[theorem]{Proposition}

\newtheorem{corollary}[theorem]{Corollary}

\newtheorem*{theorem*}{Theorem}

\theoremstyle{definition}
\newtheorem{definition}[theorem]{\rm\bf Definition}

\theoremstyle{remark}
\newtheorem{remark}[theorem]{\rm\bf Remark}

\newtheorem{example}[theorem]{\rm\bf Example}

\def\half#1#2{\begin{matrix}\frac{#1}{#2}\end{matrix}}

\def\R#1{\mathbb{R}^{#1}}

\def\Field{\mathbf{K}}

\def\ZN{\mathbb{Z}_N^\circ}

\def\scal#1#2{\langle #1; #2 \rangle}

\def\Span#1{\langle #1\rangle}

\DeclareMathOperator{\Idm}{Idm}

\DeclareMathOperator{\trace}{tr}

\DeclareMathOperator{\Aut}{Aut}
\DeclareMathOperator{\ord}{ord}
\DeclareMathOperator{\image}{im}
\DeclareMathOperator{\spec}{spec}

\begin{document}

\title{Medial and isospectral algebras}


%\titlerunning{Medial and isospectral algebras}        % if too long for running head

\author{Yakov Krasnov}
\address{Department of Mathematics, Bar-Ilan University, Ramat-Gan, 52900, Israel}
\email{krasnov@math.biu.ac.il}
\author{Vladimir G. Tkachev}
%\thanks{V.~Tkachev has been paritally supported by Stiftelsen GS Magnusons fond, grant MG2018-0042}
%\address{Department of Mathematics, Link\"oping University, Link\"oping, 58183, Sweden}
\email{vladimir.tkatjev@liu.se}

\date{\today}

\maketitle

\begin{abstract}
The purpose of this paper is to give a systematic study of two new classes of commutative nonassociative algebras, the so-called isospectral and medial algebras. An isospectral algebra $\mathbb{A}$ is a generic commutative nonassociative algebra whose idempotents have the same Peirce spectrum. A medial algebra is algebra with identity $(xy)(zw)=(xz)(yw)$. We show that these two classes are essentially coincide. We also prove that any medial spectral algebra is isomorphic to a certain isotopic deformation of the commutative associative quotient algebra $\Field[z]/(z^n-1)$.

\end{abstract}

\section{Introduction}

%Despite an enormous recent research on classification of diverse classes of (normally low-dimen\-sional) nonassociative algebras, the existing classifications usually provide  lengthy lists with structure constants, that can hardly be considered satisfactory from the conceptual point of view. Instead, it seems to be more natural to consider some ....

By an algebra $\mathbb{A}$ we shall always mean a commutative nonassociative finite dimensional algebra over a  field $\Field$ of $\mathrm{char}(\Field) \ne 2,3$. If not explicitly stated otherwise, any algebra will  be commutative (but maybe non-associative) and the algebra multiplication will normally be denoted by juxtaposition. An element $c\in \mathbb{A}$ is called an idempotent if $c^2=c$. The set of nonzero idempotents of $\mathbb{A}$ is denoted by $\Idm(\mathbb{A})$.

The \textit{spectrum} $\spec(c)$ of an idempotent $c\in \Idm(\mathbb{A})$ is the multiset of eigenvalues of the multiplication operator $L_c:x\to cx$. The \textit{Peirce spectrum} of an algebra is the union of the Peirce spectra of its nonzero idempotents.

Inspired by recent progress in axial algebras \cite{HRS15b}, \cite{Rehren17}, \cite{HSS18}, we introduce and study here two new classes of
algebras with remarkable properties. Recall that algebras with a prescribed Peirce spectrum  has been the object of extensive investigation in the context of automorphisms of  finite simple groups. In \cite{Griess82}, Griess  proved the existence of the largest simple sporadic group (Monster) by utilizing a nonassociative commutative algebra  of dimension $196883$, and showed that the automorphism group of this algebra is exactly the Monster simple group. S.~D.~Smith studied in \cite{Smith77} nonassociative commutative algebras for triple covers of 3-transposition groups and  S.~Norton considered general transposition algebras with a natural permutation representation \cite{Norton88}. Motivated by the work of Griess,  K.~Harada \cite{Harada81}, \cite{Harada84} considered the automorphism group of a vector space possessing a nonassociative commutative algebra structure, which is closely related to a natural permutation representation of a multiply transitive group. Some further results were obtained by H.P.~Allen \cite{Allen84b}, \cite{Allen84a} and K.~Narang \cite{Narang} and generalized by H.~Suzuki in \cite{Suzuki83}, \cite{Suzuki86b}. Invariant commutative nonassociative algebras for some sporadic simple groups were studied by A.J.E~Ryba \cite{Ryba96}, \cite{Ryba07}. We also mention some related classes of algebras introduced very recently in the work of D.J.F.~Fox on nonassociative algebras of metric curvature tensors \cite{Fox2021}, \cite{Fox2022}.

All the above algebras, including the Griess algebra, are generated by certain idempotents having the same Peirce spectrum with distinguished fusion rules (i.e. the multiplication rules between the corresponding Peirce eigenspaces). These ideas are captured by Majorana and \textit{axial algebras} in a recent study of A.~Ivanov, J.~Hall, F.~Rehren, S.~Shpectorov in  \cite{Ivanov09}, \cite{HRS15}, \cite{HRS15b}, \cite{Rehren17}, \cite{Ivanov2018}. Note that many of axial algebras are \textit{metrized}, i.e.  carry an \textit{invariant} bilinear form $b(x,y)$, in other words, there holds
\begin{equation}\label{metrized}
b(xy,z)=b(x,yz) \quad \text{for all $x,y,z\in \mathbb{A}$.}
\end{equation}

Beside the group theoretic context, the algebras with prescribed spectral properties also naturally appear in finite geometries, algebraic combinatorics \cite{DeMedts2018}, \cite{DeMedts17}, and differential geometry. We especially mention the class of commutative algebras associated with cubic minimal cones  which is relevant in the context of isospectral algebras. More precisely, one has the following definition \cite[Chapter~6]{NTVbook}, \cite{Tk14}, \cite{Tk18e}, \cite{Tk19a}: a commutative metrized nonassociative algebra $\mathbb{A}$ over $\R{}$ is called a \textit{Hsiang algebra} if it satisfies the trace free condition $\trace L_x=0$ and the splitting identity
\begin{equation}\label{Hsiang}
b(x^2,x^3)=k \,b(x,x)b(x^2,x),\qquad \forall x\in \mathbb{A},
\end{equation}
where $k\in \R{}$ is some fixed nonzero constant and $b(x,y)$ is a nondegenerate invariant bilinear form on $\mathbb{A}$. It follows immediately from \eqref{Hsiang} that all idempotents in a Hsiang algebra have the same length $b(e,e)=1/k$.  A more delicate property is that \eqref{Hsiang} implies that in any Hsiang algebra all idempotents have the same  spectrum with eigenvalues in $\{-1,-\frac12,\frac12,1\}$, where the multiplicity of each eigenvalue is independent on a choice of an  idempotent \cite{NTVbook}.

In the light of the above, the following definition is natural.

\begin{definition}
A commutative nonassociative algebra is called \textit{isospectral} if  all nonzero algebra idempotents have the same  spectrum.
\end{definition}

By the definition, the Peirce spectrum of an isospectral algebra coincides with the spectrum of any idempotent of the algebra.
In this paper we shall primarily concern with \textit{generic} isospectral algebras over algebraically closed fields, i.e. those containing the maximal possible finite number of distinct idempotents ($=2^{\dim \mathbb{A}}$); see section~\ref{sec:prelim} for concise definitions. Note, that the Hsiang algebras above are isospectral but they normally are \textit{not} generic and  contain infinitely many idempotents if the dimension $\ge 4$. A classification of all non-generic commutative nonassociative isospectral algebras is more involved and will be considered elsewhere.

Coming back to the generic case,  we have proved in \cite{KrTk18a} that any finite dimensional isospectral generic algebra $\mathbb{A}$ over $\mathbb{C}$ necessarily has the \textit{cyclotomic Peirce spectrum}, i.e. the finite set generated by a primitive root of unity of degree $n=\dim \mathbb{A}$.
%Equivalently, for any idempotent $c\in \Idm(\mathbb{A})$, the minimal polynomial of the multiplication operator $L(c)$ is given explicitly by $t^n-1$.

\begin{remark}
The latter implies that except for the case of dimension $\dim \mathbb{A}=2$, the spectrum of any isospectral generic algebra over $\mathbb{C}$ contains complex numbers, therefore\textit{ a isospectral generic algebra of dimension $\ge3$ is \textit{not} metrized}! Indeed, a metrized algebra must have a real spectrum because the multiplication operator $L(x)$ is self-adjoint for any element $x\in \mathbb{A}$) \cite{Tk15b}, \cite{Tk18a}.
\end{remark}

%\subsection{Two  examples of isospectral algebras}
%\label{sec:ex}

Below we consider two motivating examples of isospectral  algebras, in dimension 2 and 3 respectively.

\begin{example}\label{ex1}
Let $\mathbb{A}_2(\Field)$ be the vector space over a field $\Field$ of $\mathrm{char} (\Field)\ne2,3$ generated by two vectors $c_1$ and $c_2$. Let us introduce a commutative algebra structure on $\mathbb{A}_2(\Field)$ by requiring
\begin{equation}\label{I2}
c_1c_1=c_1, \quad c_2c_2=c_2, \quad c_1c_2=c_2c_1=-c_1-c_2.
\end{equation}
This immediately implies that $c_1$ and $c_2$, as well as $c_3=-c_1-c_2$ are distinct idempotents in $\mathbb{A}_2(\Field)$. Furthermore, an easy verification reveals that under the made assumptions, there exists exactly $3=2^2-1$ nonzero idempotents in $\mathbb{A}_2(\Field)$ and they are closed under multiplication:
\begin{equation}\label{cicjck}
c_ic_j=c_k, \quad \text{where $\{i,j,k\}=\{1,2,3\}$}.
\end{equation}
The latter identity also implies that the  spectrum  of each idempotent $c_i$ is $\{1,-1\}$, therefore $\mathbb{A}_2(\Field)$  is an isospectral algebra (see also Example~\ref{ex5} below for the particular case $\Field=\mathbb{F}_7$).  Next, we have $\trace L_{c_i}=1-1=0$ for any idempotent $c_i$, thus $\trace L_x=0$ holds true for any $x\in \mathbb{A}_2$.
Furthermore, one readily verifies that any element $x\in \mathbb{A}_2$ satisfies the identity
\begin{equation}\label{B2}
x^{3}=\beta_2(x)x,
\end{equation}
where $\beta_2(x):\mathbb{A}_2\to \Field$ is a certain \textit{multiplicative} homomorphism:
\begin{equation}\label{composition}
\beta_2(xy)=\beta_2(x)\beta_2(y),\qquad \forall x,y\in \mathbb{A}_2.
\end{equation}
Explicitly, $\beta_2$ is given by the positive definite quadratic form
$$
\beta_2(x)=x_1^2-x_1x_2+x_2^2, \quad \text{where $x=x_1c_1+x_2c_2$.}
$$
The identity \eqref{composition} also shows that $\mathbb{A}_2$ is a commutative composition \textit{nonunital} algebra \cite[\S~1.2]{SpringerVeldkamp}.
The bilinear form $b(x,y)=\beta_2(x+y)-\beta_2(x)-\beta_2(y)$ is positive definite and invariant, i.e. satisfies \eqref{metrized}. In other words, $\mathbb{A}$ is a metrized algebra, see \cite{Tk18a}.

Furthermore, by \eqref{B2}
$$
b(x^2,x^3)=b(x^2,\beta_2(x)x)=\beta_2(x)b(x^2,x)=\frac12b(x,x)b(x^2,x),
$$
which implies that $\mathbb{A}_2$  is a Hsiang algebra. In fact, $\mathbb{A}_2$ is the Hsiang algebra  of a one-dimensional minimal cone obtained by union of two perpendicular lines in $\R{2}$ (Chap.~6, \cite{NTVbook}). It is not difficult to show that $\mathbb{A}_2$ is the only isospectral generic algebra in dimension two (in any characteristic $\ne 2$).
\end{example}

\begin{remark}
The algebra $\mathbb{A}_2(\Field)$ is isomorphic to the two-dimensional Harada algebra defined \cite{Harada84}; its automorphism group is exactly the permutation group of the idempotents:
$$
\Aut(\mathbb{A}_2(\Field))=\mathrm{Perm}(\Idm(\mathbb{A}_2))\cong S_3,
$$
is isomorphic to the symmetric group of degree three. The general Harada algebras appeared very recently in the context of simplicial algebras \cite{Fox2020a} and  cyclic elements
in semisimple Lie algebras \cite{EJK}. Note, however, that the general Harada algebras in dimensions $n\ge 3$ are not isospectral.
\end{remark}

\begin{example}\label{ex2}
The above example can be appropriately generalized for dimension 3. The corresponding three dimensional isospectral algebra $\mathbb{A}_3(\mathbb{C})$ over the field of complex numbers $\mathbb{C}$ has been announced in \cite[Sec.~5.1]{KrTk18a}. More precisely, define $\mathbb{A}_3(\mathbb{C})$ as the free commutative algebra over $\mathbb{C}$ spanned by three elements $c_1,c_2,c_3$
subject to $c_ic_i=c_i$, $1\le i\le 3$, and
\begin{align*}
  c_5:=c_1c_{2}&=(\gamma-1)c_1-\gamma c_{2}+\gamma c_{3},\\
  c_6:=c_2c_{3}&=\gamma c_{1}+(\gamma-1)c_2-\gamma c_{3},\\
  c_7:=c_3c_{1}&=-\gamma c_{1}+\gamma c_{2}+(\gamma-1)c_3,
\end{align*}
where $\gamma$ is the Kleinian integer unit \cite{ConwaySmith}, i.e. a root of
$$
2\gamma^2-\gamma+1=0.
$$
Define also
$$
c_4=-\gamma(c_1+c_2+c_3).
$$
Then one can show that $c_i$, $1\le i\le 7$ are the only idempotents of  $\mathbb{A}_3$, in particular, the algebra $\mathbb{A}_3$ is generic. A straightforward examination reveals that all idempotents have the same (cyclotomic) spectrum
$$
\sigma(c_i)=\{1,\half{-1-\sqrt{-3}}2, \half{-1+\sqrt{-3}}2\}, \qquad  1\le i\le 7.
$$
Furthermore, as in Example~\ref{ex1}, one can show that all nonzero idempotents in $\mathbb{A}_3(\mathbb{C})$ are closed under the algebra multiplication, i.e. $c_ic_j\in \Idm(\mathbb{A}_3(\mathbb{C}))$ for any $1\le i,j\le 7$. More precisely, the multiplication rules  between idempotents $c_i$ follow the pattern in the left part of Table~\ref{tab1}.

\begin{table}
\begin{tabular}{c|ccccccc}
%\begin{array}{c|ccccccc}
&1&2&3&4&5&6&7\\\hline
1&1&5&7&3&6&2&4\\
2&5&2&6&1&4&7&3\\
3&7&6&3&2&1&4&5\\
4&3&1&2&4&7&5&6\\
5&6&4&1&7&5&3&2\\
6&2&7&4&5&3&6&1\\
7&4&3&5&6&2&1&7\\
%\end{array}
\end{tabular}
\qquad \qquad
\begin{tabular}{c|ccccccc}
%\begin{array}{c|ccccccc}
%&1&2&3&4&5&6&7\\\hline
%1&1&3&2&5&4&7&6\\
%2&3&2&1&6&7&4&5\\
%3&2&1&3&7&6&5&4\\
%4&5&6&7&4&1&2&3\\
%5&4&7&6&1&5&3&2\\
%6&7&4&5&2&3&6&1\\
%7&6&5&4&3&2&1&7\\
%\end{array}\qquad\qquad
%\begin{array}{c|ccccccc}
&1&2&3&4&5&6&7\\\hline
1&1&5&2&6&3&7&4\\
2&5&2&6&3&7&4&1\\
3&2&6&3&7&4&1&5\\
4&6&3&7&4&1&5&2\\
5&3&7&4&1&5&2&6\\
6&7&4&1&5&2&6&3\\
7&4&1&5&2&6&3&7\\
%\end{array}
\end{tabular}
\vspace*{0.5cm}
\label{tab1}
\caption{The left multiplication table corresponds to $c_i$ for Example~\ref{ex2}, the right table illustrates the medial law \eqref{nizerule7} after permutation \eqref{perm}}
\end{table}

For example, $c_2c_5=c_4$ and $c_3c_5=c_1$.
It is easy to see that the latter multiplication table is a \textit{Latin square}, i.e. each index occurs precisely one time in each column and row. In fact, there is another remarkable property of this Latin square: the authors were completely surprised to notice that the table also satisfies the \textit{medial magma identity}
\begin{equation}\label{magma}
(xy)(zw)=(xz)(yw), \qquad \forall x,y,z,w\in \Idm(\mathbb{A}_3).
\end{equation}
We point out, that, although it is easy to verify \eqref{magma} for  particular cases, it is quite nontrivial to examine  \eqref{magma} for \textit{all} quadruples $(x,y,z,w)$ (taking into account the commutativity, one have to verify totally $\frac12(\frac{7\cdot 6}2)\cdot ((\frac{7\cdot 6}2-1)=210$ quadruples). In fact, we will show in this paper that \eqref{magma} holds for all  isospectral algebras including $\mathbb{A}_3$. But in this particular case, we can explain how to check \eqref{magma} right now. To this end, let us consider the permutation of $c_i$ given accordingly to
\begin{equation}\label{perm}
\left(\begin{array}{cccccccc}
1&2&3&4&5&6&7\\
1&5&2&3&6&4&7
\end{array}
\right)
\end{equation}
Then the resulting multiplication  (the so-called \textit{isotopy} of the Latin square) corresponds to  the right part of Table~\ref{tab1}. This updated table is  more illuminating and informative. In particular, one can easily verify that the entries in the right table satisfy the following nice rule:
\begin{equation}\label{nizerule7}
i\circledast j=\frac12 (i+j) \mod 7,
\end{equation}
i.e. the multiplication of idempotents (after permutation \eqref{perm}) satisfies $c_ic_j=c_{i\circledast j}$.
This new multiplication $\circledast$ on $\mathbb{Z}_7$ is commutative but nonassociative. But, since
$$
(i\circledast j)\circledast(k\circledast l)\equiv \frac14 (i+j+k+l) \mod 7,
$$
the new multiplication $\circledast$ obviously satisfies the medial magma identity \eqref{magma}.

Since $\mathbb{A}_3(\mathbb{C})$ is generated by the idempotents, it  follows that any quadruple of elements of $\mathbb{A}_3(\mathbb{C})$ also satisfies \eqref{magma}. This motivates the following

\begin{definition}
A commutative algebra $\mathbb{A}$ satisfying
\begin{equation}\label{magma0}
(xy)(zw)=(xz)(yw), \qquad \forall x,y,z,w\in \mathbb{A}
\end{equation}
is called \textit{medial}.
\end{definition}

\noindent
Finally, one can show that any element of $\mathbb{A}_3(\mathbb{C})$ satisfies the identity generalizing \eqref{B2}:
\begin{equation}\label{B4}
x^{4}=\beta_3(x)x,\qquad \forall x\in \mathbb{A}_3(\mathbb{C}),
\end{equation}
where $\beta_3(x):\mathbb{A}_3(\mathbb{C})\to \mathbb{C}$ is a certain multiplicative homomorphism.
\end{example}

The remarkable medial magma property for isospectral algebras $\mathbb{A}_2(\Field)$ and $\mathbb{A}_3(\mathbb{C})$ has been one of the starting points for the research presented in this paper. The  medial magma identity itself and its variations is very well-known in the context of quasigroups, finite geometries and Steiner designs. The classical result of Toyoda \cite{Toyoda41}, Bruck \cite{Bruck44} and Murdoch \cite{Murdoch41} characterizes a medial quasigroup is an isotope of a certain abelian group. On the other hand, to the best of our knowledge, the algebras satisfying the medial magma identity have not yet been systematically studied; see, however, some related train algebras satisfying polynomial identities studied in \cite{LabraR}.

\medskip
The paper is organized as follows. We recall the basic facts about the Peirce decomposition, generic algebras and quasigroups in section~\ref{sec:prelim}. In section~\ref{sec:main} we  formulate our main results. In section~\ref{sec:medial} we discuss medial algebras and their Peirce decomposition. In particular, we show that any medial algebra is a direct product of its radical and an isospectral algebra. This establishes a link between medial and isospectral algebras. Next, isospectral algebras are discussed in section~\ref{sec:iso} and the subclass of medial isospectral algebras in section \ref{sec:med-iso}.   In section~\ref{sec:polynom} we prove the existence of a medial isospectral algebra in any dimension $n\ge 2$. The set of all idempotents in a isospectral medial algebra has a natural structure of a (medial) qiuasigroup.  The classical result of Toyoda-Bruck-Murdoch characterizes a medial quasigroup is an isotope of a certain abelian group. In fact, we establish in section~\ref{sec:structure} that the idempotent quasigroup of isospectral medial alegrbas is much simpler: it is just cyclic.

\medskip
For the reader convenience, we list below some frequently used notation:
\begin{tabbing}
\hspace{95pt}\= \= \\
$\mathbb{A}$\> a commutative nonassociative algebra over a field $\Field$\\
$L_y:x\to yx$\> the (left) multiplication operator by $y$\\
$\mathbb{A}_c(\lambda)$\> the $\lambda$-eigenspace of $L_c$ \\
$\Idm(\mathbb{A})$\> the set of nonzero idempotents of $\mathbb{A}$\\
$\spec(c)$\> the spectrum of an idempotent $c$, i.e. the multiset of eigenvalues of $L_c$\\
$\sigma(c)$\> the Peirce spectrum of $c$, i.e. the distinct eigenvalues of $L_c$\\
$\epsilon_n:=e^{\frac{2\pi \sqrt{-1}}{n}}$\>
the primitive root of unity order $n$.
\end{tabbing}

\section{Main results}
\subsection{Preliminaries}\label{sec:prelim}
%\subsection{The Peirce decomposition}
%\subsection{Preliminaries}
We use the standard notation for nonassociative monomials generated by $x\in \mathbb{A}$: $x^2=xx$, $x^3=x(xx)$. Note that in general $x^4=xx^3$ and $x^2x^2$ maybe different.

We recall some basic terminology and concepts following to \cite{JacobsonBook}, \cite{KrTk18a}, \cite{Tk18e}. An element $c$ of an algebra $\mathbb{A}$ is called an idempotent (resp. a $2$-nilpotent) if $c^2=c$ (resp. $c^2=0$). By $\Idm(\mathbb{A})$ we denote the set of all nonzero idempotents of $\mathbb{A}$. Sometimes it is convenient to denote by $\Idm_0(\mathbb{A})$ the set  of all idempotents (including zero). Given an idempotent $c$, one can define the (left$=$right) multiplication operator
$$
L_c:x\to cx.
$$
We abuse terminology by saying that a vector $x\ne 0$ is an \textit{eigenvector} of an idempotent $c$ with eigenvalue $\lambda$ if  $cx=L_cx=\lambda x$. Any eigenvalue $\lambda$ of $c$ is a root of the characteristic polynomial of  $L_c$ regarded  as an endomorphism $L_c\in \mathrm{End}_\Field(\mathbb{A})$.

\begin{definition}
The \textit{spectrum of an idempotent} $c$, denoted by $\spec (c)$, is  the multi-set (a set with repeated elements) of all roots of the characteristic polynomial of $L_c$. We also denote by $\sigma(c)$ the \textit{Peirce spectrum} of $c$, i.e. the set of all \textit{distinct} eigenvalues of $L_c$. In other words, the Peirce spectrum is the support of $\spec (c)$. Thus, two idempotents may have distinct spectra while their Peirce spectra may be the same. By
$$
\sigma(\mathbb{A})=\bigcup_{c\in \Idm(\mathbb{A})}\sigma(c)
$$
we denote the \textit{Peirce spectrum of an algebra} $\mathbb{A}$, i.e. all possible distinct eigenvalues of $L_c$, when $c$ runs over all idempotents in $\Idm(\mathbb{A})$.
\end{definition}

The classical example is the projection operator with (Peirce) spectrum $0$ and $1$, where the multiplicity of $1$ indicates the dimension of the projection target space.

The  spectrum of any idempotent is always nonempty because $1\in \sigma(c)$. An idempotent $c$ is called primitive if the eigenvalue $1$ has multiplicity one.
An idempotent $c$ is called \textit{semi-simple} if $\mathbb{A}$ splits into a direct sum of  $\lambda$-eigenspaces $\mathbb{A}_c(\lambda)$ of $L_c$, where $\lambda\in \sigma(c)$:
\begin{equation}\label{Peircede}
\mathbb{A}=\bigoplus_{\lambda\in \sigma(c)}\mathbb{A}_c(\lambda).
\end{equation}
This decomposition is also known as the \textit{Peirce decomposition} of $\mathbb{A}$ relative to $c$. A \textit{fusion law} is a map $\star:\sigma(c)\times \sigma(c)\to 2^{\sigma(c)}$  such that
\begin{equation}\label{fusionlaws}
\mathbb{A}_c(\lambda)\mathbb{A}_c(\mu)\subset \bigoplus_{\nu\in \lambda\star \mu}\mathbb{A}_c(\nu).
\end{equation}

Recall that two algebras $\mathbb{A}$ and $\mathbb{B}$  over a field $\Field$ are \textit{isotopic} if there exist three non-singular linear transformations $f$, $g$ and $h$ from $\mathbb{A}$ to $\mathbb{B}$ such that
\begin{equation}\label{fgh}
f(x)g(y)=h(xy).
\end{equation}
The triple $(f,g,h)$ is called isotopism between the algebras $\mathbb{A}$ and $\mathbb{B}$. If $f=g$, $\mathbb{A}$ and $\mathbb{B}$ are called strongly isotopic, or $\mathbb{A} \simeq\mathbb{B}$. If $f=g=h$, the algebras $\mathbb{A}$ and $\mathbb{B}$ are called isomorphic, or $\mathbb{A} \cong\mathbb{B}$.

The isotopism of algebras is an important concept of nonassociative algebra introduced by \cite{AlbertIsotop}; see also \cite{Falcon17}.  Partition of nonassociative algebras into (strong) isotopism classes is very different from that the partition into classes of non-isomorphic algebras, but it helps to recognize some very common patterns which are usually not seen via isomorphisms. We refer the interested reader to a nice survey of R.~Falc\'{o}n, and O.~Falc\'{o}n, and N\'{u}\~{n}ez \cite{FalconIso} about a comprehensive discussion of the isotopy concept and its  applications.

%\subsection{Generic algebras}
An important concept in the context of the present paper is \textit{generic}  nonassociative algebras. This class has been recently introduced and thoroughly studied in \cite{KrTk18a}. Below we recall some basic facts and  refer the interested reader to \cite{KrTk18a} for more details and the  proof of Theorem~\ref{th:gener} below.

The definition of a generic algebra is based on the Segre's classical approach ~\cite{Segre} to  idempotents in commutative algebras over the complex numbers (or, in general, over an algebraically closed field $\Field$) by purely algebraic geometry methods. More precisely, Segre remarked that the set of idempotents of an algebra can be characterized as the solution set of a system of quadratic equations over $\Field$. Since a generic (in the Zariski sense) polynomial system has always \textit{B\'ezout's number} of solutions, this implies that a {generic}  algebra must have exactly $2^{\dim A}$ distinct idempotents. The genericity here should be understood in the sense that the subset of nonassociative algebra structures on a vector space $V$ is an open Zariski subset in $V^*\otimes V^*\otimes V$. This suggests the following formal definition.

 Given an algebra $\mathbb{A}$ over  a subfield $\Field$ of complex numbers, by $\mathbb{A}_\mathbb{C}$ we denote the complexification of $\mathbb{A}$ obtained in an natural way by extending  the ground field to $\mathbb{C}$.

\begin{definition}
An algebra $\mathbb{A}$ over $\Field$  is called a \textit{generic nonassociative algebra} if its comlexification $\mathbb{A}_{\mathbb{C}}$ contains exactly $2^n$ distinct idempotents, where $n=\dim \mathbb{A}$.
\end{definition}

\begin{theorem}[Theorem~3.3 in \cite{KrTk18a}]\label{th:gener}
If $\mathbb{A}$ is a commutative nonassociative generic algebra  then $\half12\not\in\sigma(\mathbb{A})$. In the converse direction: if $\half12\not\in\sigma(\mathbb{A})$ and $\mathbb{A}$ does not contain $2$-nilpotents then $\mathbb{A}$ is generic.
\end{theorem}

The main result of \cite{KrTk18a} states that idempotents in a generic commutative algebra must satisfy certain a priori constrains, \textit{syzygies}. This phenomenon is somewhat unexpected taking into account that an algebra in the context is generic.  Remarkably, the syzygies can be written explicitly. We combine the results of Theorem~4.1 and Proposition~4.2 in \cite{KrTk18a} in the theorem below.

\begin{theorem}\label{th4}%[Nonsingular case]
Let $\mathbb{A}$ be a generic commutative nonassociative algebra over $K$, $\dim A=n$.  Then
\begin{equation}\label{polynom}
\sum_{c\in \Idm(\mathbb{A})}\frac{\chi_c(t)}{\chi_c(\half12)}=2^n(1-t^n),\quad \forall t\in\R{},
\end{equation}
where $\chi_c(t)=\det (L_c-\mathbf{1}t)$ is the characteristic polynomial of $L_c$.
Furthermore, let $H(x):\Field^n\to \Field^s$ be a vector-valued polynomial map ($s\ge 1$) such that for each coordinate $\deg H_i\le n-1$, $1\le i\le n$. Then
\begin{align}
\sum_{c\in \Idm_0 (\mathbb{A})}\frac{H(c)}{\chi_c(\frac12)}&=0,\label{EuJa4}
\end{align}
In particular,
\begin{align}
\sum_{c\in \Idm(\mathbb{A})}(\chi_c(\frac12))^{-1}c&=0,\label{EuJa2}
\end{align}

\end{theorem}

For the reader convenience, we recall the key steps of the proof. Since $\mathbb{A}$ is a commutative algebra over an infinite field, the multiplication is uniquely determined by the square mapping
$\Psi(x)=x^2:\mathbb{A}\to \mathbb{A}$ by polarization
$$
xy=\frac12(\Psi(x+y)-\Psi(x)-\Psi(y))=\frac12 D\Psi(x)\,y
$$
where $D\Psi(x)$ is the Jacobi matrix of $\Psi$ at $x$.
In this notation, idempotents $c\in\mathbb{A}$ are in one-to-one correspondence with zeroes of the quadratic map $F(x)=\Psi(x)-x$. Note that
$$
D\Psi(x)=2L_x.
$$
Then the genericity condition on $\mathbb{A}$ is equivalent to saying that all roots of $F(x)=0$ are nonsingular, i.e. $DF(x)$ has the maximal rank. Applying the Euler-Jacobi formula \cite[p.~106]{ArnVarG} (see also Theorem~4.3 in \cite{BKK12}), one obtains
$$
0=\sum_{c: F(c)=0}\frac{H(c)}{\det[DF(c)]}=
\sum_{c: \Psi(c)=c}\frac{H(c)}{\det[D\Psi(c)-1]}=
\sum_{c\in\Idm(\mathbb{A})}\frac{H(c)}{2^n\det[L_c-\frac12]}
$$
which proves \eqref{EuJa4}. Applying \eqref{EuJa4} to $H(x)=$ the $i$th coordinate function (in a certain fixed basis of $\mathbb{A}$) one arrives at \eqref{EuJa2}.

%\subsection{Quasigroups}\label{sec:quasi}
Finally, let us recall some principal facts from quasigroup theory. Following \cite{Belousov}, \cite{Cameron}, we define a \textit{quasigroup}  to be a set $Q$ with a binary operation $\circ$ in which the equations $a\circ x = b$ and $y\circ a = b$ have unique solutions $x$ and $y$ for any given $a$ and $b$ from $Q$. A \textit{Latin square} is an $N \times N$ array containing $N$ different entries, such that each entry occurs exactly once in each row and once in each column. Thus, a set with a binary operation is a quasigroup if and only if its operation table is a Latin square. Two quasigroups $(Q, \circ)$ and $(S,\bullet)$ are called \textit{isotopic} if there exist bijections $f,g,h:Q\to S$ such that $f(x)\bullet g(y)=h(x\circ y)$ for any $x,y\in Q$.

The number of elements of a quasigroup is called its order and denoted by $|Q|$. A quasigroup (Latin square) $(Q,\circ)$ is called \textit{commutative}  (resp. \textit{idempotent}) if $x\circ y=y\circ x$ (resp. the identity $x\circ x=x$ holds for any $x\in Q$). Similarly, $Q$ and it is called \textit{medial} \cite[p.~33]{Belousov} if the identity
\begin{equation}\label{xyzw}
(x\circ y)\circ(z\circ w)=(x\circ z)\circ(y\circ w)
\end{equation}
holds. Note that the concept of a quasigroup with the medial identity was originally introduced by Murdoch \cite{Murdoch41} by the name abelian quasigroups; see also \cite{Bruck44}, \cite{Toyoda41}. We refer also to \cite{Stein57} for some equivalent descriptions of medial quasigroups and their structure.
Many examples of medial semigroups as well as their discussion in the context of abstract means and mid-points can be found in \cite{Frink55}.

An important theorem of Bruck-Murdoch-Toyoda   classifies medial quasigroups.
\begin{theorem}[The Bruck-Murdoch-Toyoda Theorem \cite{Murdoch41}, \cite{Bruck44}, \cite{Toyoda41}]\label{th:murd}
For any medial quasigroup $(Q,\circ)$, there exists an abelian group $(G, +)$, a fixed element $g\in G$, and commuting automorphisms $\phi,\psi:G\to G$ such that $x\circ y=g+\phi(x)+\psi(y)$.
\end{theorem}

This implies that every medial quasigroup is \textit{isotopic} to an abelian group.  For the proof,  see \cite[Theorem~2.10]{Belousov} and also \cite{Shcherbacov05} for further developments.

%In this paper, we shall primarily  concern with idempotent commutative quasigroups. Let  $Q$ be an ICQ and $x\in Q$. Define the multiplication operator $L_x:x\to xy=yx$. By the quasigroup definition, $L_x$ is invertible for any $x$.

\subsection{Main results}\label{sec:main}

Medial algebras constitutes a rather big class. In this paper, we completely classify one particular case. More precisely, we shall prove the following result.

\begin{theorem}%[Theorem \ref{the:uniq}]
For a fixed dimension $n\ge2$, there exists exactly one isomorphy class of medial generic  isospectral algebras over an algebraically closed field $\Field$. Furthermore, any medial generic  isospectral algebra  $\mathbb{A}$ satisfies the identity
\begin{equation}\label{Bn}
x^{n+1}=\beta_n(x)x,\qquad x\in\mathbb{A},
\end{equation}
where $n=\dim \mathbb{A}$ and $\beta_n:\mathbb{A}\to \Field$ is a multiplicative homomorphism.
\end{theorem}

The key ingredient of the uniqueness part of the above theorem is the  Peirce decomposition of a medial isospectral algebra. Given an idempotent $c\in \Idm(\mathbb{A})$ we prove that  there exist $w_1\in\mathbb{A}$ such that  $\{w_1^i\}_{i=0}^{n-1}$ (where $w_1^0=c$) and the algebra $\mathbb{A}$ decomposes into the direct sum of one-dimensional eigespaces spanned on the principal powers $w_1^i$ (for the definition, see \eqref{princip} below):
$$
A=\bigoplus_{i=0}^{n-1}\Span{w_1^i}, \qquad cw_1^i=\epsilon^{i}w_1^i.
$$
Then the multiplication between the eigenvectors is given by
\begin{equation}\label{kronecker0}
w_1^i w_1^j=\epsilon_n^{\delta_{ij,0}-(i-1)(j-1)}w_1^{i+j},
\end{equation}
where $\delta_{i,j}$ is the Kronecker delta. The latter relations  determine the multiplication structure on a medial isospectral algebra up to an isomorphism.

The existence part follows from the following result which we prove in Section~\ref{sec:polynom}.
%of a medial isospectral algebra in any dimension $n$.

\begin{theorem}%[Theorem \ref{the:model}]
Let $n\ge 2$. Consider the quotient
$$
\mathscr{C}_n=\Field[z]/(z^n-1)
$$ with  a new multiplication $\circ$ on $\mathscr{C}_n$ defined by
\begin{equation}\label{mmul}
p(z)\circ q(z)=p(\epsilon_n z)q(\epsilon_n z) \mod z^n-1,
\end{equation}
where $\epsilon_n$ is a primitive root of unity of order $n$. Then the algebra $\mathscr{C}_n$ is a medial isospectral algebra over $\mathbb{C}$ of dimension $n$. Furthermore, the corresponding multiplicative homomorphism \eqref{Bn} is given explicitly by
\begin{equation}\label{PN}
\beta(p(z))=p(1)p(\epsilon_n)p(\epsilon_n^2 )\cdots p(\epsilon_n^{n-1}):\mathscr{C}_n\to \Field.
\end{equation}
\end{theorem}

\subsection{Idempotents as a quasigroup}
In other words, any medial isospectral algebra is isomorphic to some $\mathscr{C}_n$. It is interesting to study the idempotents of such algebras.
While the algebra structure on $\mathscr{C}_n$ and the multiplicative homomorphism $\beta(z)$ are given very  explicitly by \eqref{mmul} and \eqref{PN}, the structure of idempotents of $\mathscr{C}_n$ is more sophisticated. When $n\le 6$, it is possible to describe the structure of $\Idm(\mathscr{C}_n)$ is a straightforward way. In order to understand the relations between idempotents in a  satisfactory way, one need to look at $\Idm(\mathscr{C}_n)$ as a quasigroup.

To this end, note that it follows from \eqref{magma} and the fact that the Peirce spectrum of an isospectral algebra does not contain $0$ that the product of two nonzero  idempotents  in a medial algebra $\mathbb{A}$ is always a nonzero idempotent  and the relation $xy=z$ uniquely determines one  idempotent if two others are given. This implies that $\Idm(\mathbb{A})$ is a commutative idempotent medial (CIM, for short) quasigroup with respect to the algebra multiplication. Then the classical  Bruck-Murdoch-Toyoda theorem characterizes any medial quasigroup as an isotope of a certain abelian group $G$. The structure of a CIM quasigroups and their endomorphisms has been recently studied by A.~Leibak and P.~Puusemp \cite{Leibak2014}, \cite{LeibakPuusemp}. In \cite{Leibak2014} a refinement of the Bruck-Murdoch-Toyoda theorem for CIM quasigroups has also been established. A \textbf{natural question arise}: Given a medial isospectral algebra $\mathbb{A}$, how to characterize the underlying  abelian group $G$? We establish that $G$ must be the cyclic group $\mathbb{Z}_{2^n-1}$, $n=\dim \mathbb{A}$.

More precisely, we have

\begin{theorem}\label{the:Znn}
Let $\mathbb{A}$ be a medial isospectral algebra of dimension $n$. Then the there exists a bijection $\phi:\Idm(\mathbb{A})\to \mathbb{Z}_{2^n-1}$  such that
$$
\phi(xy)\equiv \frac{1}{2}(\phi(x)+\phi(y))\mod 2^n-1.
$$
\end{theorem}

\section{Medial algebras}\label{sec:medial}

\subsection{General properties}
Obviously, a zero algebra (i.e. an algebra with $\mathbb{A}\mathbb{A}=0$) is medial. To avoid this trivial case, we shall always assume that all algebras below are nonzero.

Let $\mathbb{A}$ be a medial algebra. Then the specialization $x=z$ and $y=w$ in \eqref{magma} implies
\begin{equation}\label{medial0}
(xy)^2=x^2y^2, \qquad \forall x,y \in\mathbb{A}.
\end{equation}

It turns out that the identity \eqref{medial0} is essentially equivalent to the medial magma identity \eqref{magma}. More precisely, we have

\begin{proposition}\label{pro:med}
If $\Field$ is a field of characteristic not $2$ or $3$ then \eqref{medial0} and \eqref{magma} are equivalent.
\end{proposition}

\begin{proof}
It suffices to show that \eqref{medial0} implies \eqref{magma}. To this end, note that the linearization of \eqref{medial0} yields
\begin{equation}\label{medial0a}
(xy)(xz)=x^2(yz), \qquad \forall x,y,z \in\mathbb{A},
\end{equation}
and a further linearization in $x$ yields
\begin{equation}\label{medial0b}
(xy)(wz)+(wy)(xz)=2(xw)(yz), \qquad \forall x,y,z,w \in\mathbb{A}.
\end{equation}
Permuting the variables we obtain
\begin{align*}
(xz)(wy)+(wz)(xy)&=2(xw)(yz)\\
(xz)(yw)+(yz)(xw)&=2(xy)(wz),
\end{align*}
and summing up the three latter identities yields
$$
2(xy)(wz)+3(wy)(xz)+(yz)(xw)=4(xw)(yz)+2(xy)(wz),
$$
hence $3(wy)(xz)=3(xw)(yz),$ which implies  \eqref{magma} (after a suitable renaming of variables).
\end{proof}

Thus, when working with algebras over a field with characteristic $2$ or $3$ one need to be careful with the defining identities. In this paper, we work with $\mathrm{char}(\Field) \ne 2,3$, therefore we shall not distinguish the definitions.

The identity  \eqref{medial0} is remarkable in several aspects. First note that it is equivalent to that the square mapping $\Psi(x)=x^2$ is a multiplicative homomorphism and (under various names) has attracted much attention, primarily in the context of quasigroups. Furthermore, algebras with \eqref{medial0} appear, for example, in the study of train algebras satisfying polynomial identities \cite{LabraR}.  Also, we mention that \eqref{medial0} resembles the Norton inequality
$
\scal{x^2}{y^2}\ge \scal{xy}{xy}
$
which appears as an axiom in construction of the Griess-Conway-Norton algebra of the monster sporadic simple group and also in general  Majorana algebras, see \cite{MeyerNeutsch}, \cite{Ivanov09}. Note also that the latter inequality is defined in the presence of an invariant bilinear form $\scal{}{}$.

We have the following elementary corollaries of the definitions.

\begin{proposition}
\label{pro:suff}
A subalgebra of a medial algebra is medial. Furthermore, if some basis $\{e_i\}_{i=1}^n$ of a commutative algebra $\mathbb{A}$ then $\mathbb{A}$ is medial if and only if $\{e_i\}_{i=1}^n$ satisfies \eqref{magma}.
\end{proposition}

\begin{proof}
The first claim is obvious. For the second claim it suffices to verify the `if' condition. To this end, suppose that a basis $\{e_i\}_{i=1}^n$ of a commutative algebra  $\mathbb{A}$ satisfies \eqref{magma}. Then for any quadruple  $(x,y,z,w)$ we have
\begin{align*}
(xy)(zw)&=\sum_{i,j,k,l=1}^n x_iy_jz_kw_l(e_ie_j)(e_ke_l)=\sum_{i,j,k,l=1}^n x_iy_jz_kw_l(e_ie_k)(e_je_l)=(xz)(yw)
\end{align*}
where $x=\sum_{i=1}^n x_ie_i$ etc. Hence $\mathbb{A}$ satisfies \eqref{magma}, as desired.
\end{proof}

%\subsection{Semigroup of squares}

\begin{definition}
Let $\mathbb{A}$ be a medial algebra. Define
$$
\mathbb{A}^{\Box} :=\{x\in \mathbb{A}: \quad \text{there exists $a\in \mathbb{A}$ with $x=a^2$}\}.
$$
\end{definition}
An immediate observation is
\begin{equation}
\label{square}
\Idm(\mathbb{A})\subset \mathbb{A}^{\Box}.
\end{equation}

\begin{proposition}
For any medial algebra $\mathbb{A}$, the set  $\mathbb{A}^{\Box}$ is a nontrivial multiplicative magma.
\end{proposition}

\begin{proof}
Since $\mathbb{A}^{\Box}$ is nonzero algebra, there exists $a\in \mathbb{A}$ such that $a^2\ne 0$. Then $x=a^2\in \mathbb{A}^{\Box}$. Therefore $\mathbb{A}^{\Box}$ is nonempty. If $x,y\in \mathbb{A}^{\Box}$ then $x=a^2$, $y=b^2$, therefore $xy=(ab)^2\in \mathbb{A}^{\Box}$, i.e. $G$ is a multiplicative magma.
\end{proof}

\subsection{The quasigroup property}
The set of idempotents in a medial algebra, if nonempty, has a very special structure.

\begin{proposition}
\label{pro:idem}
If $c_1,c_2$ are idempotents in a medial algebra then so is $c_1c_2$ (maybe zero). In other words, the set of all idempotents in a medial algebra is a multiplicative magma.
\end{proposition}

\begin{proof}
Follows immediately from \eqref{medial0}.
\end{proof}

\begin{proposition}
\label{pro:endom}
Let $\mathbb{A}$ be a medial algebra.  For any idempotent $c\in \Idm(\mathbb{A})$:
\begin{enumerate}[label=(\alph*)]
\item\label{ideal1}
$L_c(xy)=L_c(x)L_c(y)$ for all $x,y\in \mathbb{A}$, i.e. the multiplication operator $L_c:\mathbb{A}\to \mathbb{A}$ is an algebra endomorphism;
\item\label{ideal2}  $\mathbb{A}_c(0)$ is an ideal of $\mathbb{A}$;
\item\label{ideal3} $L_c(\mathbb{A})$ is a subalgebra of $\mathbb{A}$;
\item\label{ideal4} $\mathbb{A}_c(1)$ is a  subalgebra of $L_c(\mathbb{A})$ and $\dim \mathbb{A}_c(1)\ge 1$;
\item\label{ideal5} for any idempotents $c_1, c_2\in \Idm(\mathbb{A})$ the following transformation rule holds:
\begin{equation}\label{c1c2}
L_{c_2}L_{c_1}=L_{c_2c_1}L_{c_2}
\end{equation}
\end{enumerate}
\end{proposition}

\begin{proof}
\ref{ideal1}: the multiplication operator $L_c:\mathbb{A}\to \mathbb{A}$ is linear and it follows from the medial magma identity that for any idempotent $c\in \Idm(\mathbb{A})$
\begin{align*}
L_c(xy)&=c(xy)=(cc)(xy)=(cx)(cy)=L_cx\,L_cy,
\end{align*}
hence $L_c$ is an algebra endomorphism. As the kernel of a homomorphism, $\mathbb{A}_c(0)=\ker L_c$ is ideal and as the image of a homomorphism, $\image L_c=L_c(\mathbb{A})$ is a subalgebra, which proves \ref{ideal2} and \ref{ideal3} respectively. Further, $\mathbb{A}_c(1)=\{x:L_cx=x\}$ is the set of fixed points of $L_c$, thus, is a subalgebra of $\mathbb{A}$. Since $L_c$ stabilizes $\mathbb{A}_c(1)$, the latter is also a subalgebra of $L_c(\mathbb{A})$. Also,
$$
\Span{c}\subset \mathbb{A}_c(1)\subset cA,
$$
hence $\dim \mathbb{A}_c(1)\ge 1$. This proves \ref{ideal4}. Finally, \ref{ideal5} follows from
$$
L_{c_2}L_{c_1}x=c_2(c_1x)=(c_2c_2)(c_1x)=(c_2c_1)(c_2x)=L_{c_2c_1}L_{c_2}x.
$$
\end{proof}

The set of idempotents with non-degenerated multiplication operator
$$
\Idm^{\times}(\mathbb{A}):=\{c\in \Idm(A): \, \ker L_c=0 \}\subset \Idm(\mathbb{A})
$$
plays a distinguished role.

\begin{proposition}
\label{pro:quasi}
$\Idm^{\times}(\mathbb{A})$ is a medial idempotent quasigroup. Furthermore, all $c\in \Idm^{\times}(\mathbb{A})$ have the same characteristic polynomial.
\end{proposition}

\begin{proof}
If $c_1,c_2\in \Idm^{\times}(\mathbb{A})$ then $c_1c_2$ is a nonzero idempotent, hence $\Idm^{\times}(\mathbb{A})\Idm^{\times}(\mathbb{A})\subset \Idm(\mathbb{A})$. Furthermore, since $L_{c_1}$ is invertible,  by \ref{ideal5} in Proposition~\ref{pro:idem} we obtain
\begin{equation}\label{lc1c2}
L_{c_1c_2}=L_{c_1}L_{c_2}L_{c_1}^{-1}
\end{equation}
is also invertible, hence in fact $c_1c_2\in \Idm^{\times}(\mathbb{A})$. Therefore, $\Idm^{\times}(\mathbb{A})\Idm^{\times}(\mathbb{A})\subset \Idm^{\times}(\mathbb{A})$ i.e. $\Idm^{\times}(\mathbb{A})$ is a multiplicative magma. Furthermore, if $c_2c_1=c_3c_1$ for some $c_i\in \Idm^{\times}(\mathbb{A})$ then $c_1(c_2-c_3)=L_{c_1}(c_2-c_3)=0$, hence $c_2-c_3=0$. This  yields that $\Idm^{\times}(\mathbb{A})$ is a quasigroup. Since $\Idm^{\times}(\mathbb{A})$ satisfies \eqref{magma}, it is also a medial idempotent quasigroup.  Next, it follows from \eqref{lc1c2} and $c_1c_2=c_2c_1$ that
$$
L_{c_1}L_{c_2}L_{c_1}^{-1}=L_{c_2}L_{c_1}L_{c_2}^{-1},
$$
i.e. $L_{c_1}$ and $L_{c_2}$ are conjugate linear endomorphisms, thus have the same characteristic polynomials. The proposition follows.
\end{proof}

\subsection{Medial and associative algebras}
There are useful connections between medial and associative algebras. In one direction, the relation is simple. Indeed, if  $\mathbb{A}$ is associative and commutative then
$(xy)(zw)=xyzw=xzyw=(xz)(yw)$, hence

\begin{proposition}\label{pro:ex3}
Any associative commutative algebra is medial.
\end{proposition}

In the converse direction, one has the following  observation. If $A$ is a unital medial algebra and $e$ is the algebra unit then
$$
x(yz)=(ex)(yz)=(ez)(xy)=z(xy)=(xy)z.
$$
This implies

\begin{proposition}
\label{pro:unital}
If $\mathbb{A}$ is a unital medial algebra then  $\mathbb{A}$ is  associative.
\end{proposition}

\begin{remark}
The above observations show that medial algebras are `nearly associative'.
Of course, a general medial algebra is not associative. The simplest example is the two-dimensional  algebra $\mathbb{A}_2$ from Example~\ref{ex1} above. Indeed, the identity \eqref{cicjck} implies that the basis $\{c_1,c_2\}$ satisfies \eqref{magma}, thus by Proposition~\ref{pro:suff} $\mathbb{A}_2$ is medial. Since
$$(c_1c_2)c_3=c_3c_3=c_3\ne c_1=c_1c_1=c_1(c_2c_3),
$$
$\mathbb{A}_2$ is nonassociative.
\end{remark}

\begin{remark}
In the above context, a natural question arises: Does a unitalization preserve the medial property? The answer is `no'. Indeed, let  $\mathbb{A}$ be any nonassociative  medial algebra. Then its unitalization $\mathbb{A}^{\mathrm{uni}}$ contains $\mathbb{A}$ as a subalgebra, therefore $\mathbb{A}^{\mathrm{uni}}$ is also nonassociative. But by Proposition\eqref{pro:unital}, any unital medial algebra is associative, hence $\mathbb{A}^{\mathrm{uni}}$ in this case is \textit{not} medial.
\end{remark}

On the other hand, if a medial algebra has an idempotent with nondegenerate multiplication (i.e. the set $\Idm^{\times}(\mathbb{A})$ is non-empty) then such an algebra is a (strong) isotope of a  commutative associative unital algebra. More precisely, using Kaplansky's Trick \cite{Peter2000}, \cite{Kaplansky} one obtains

\begin{proposition}
\label{pro:ass}
Let $\mathbb{A}$ be a medial algebra, $c\in \Idm^{\times}(A)$. Define  the strong isotopy
 $(\mathbb{A},\circ)$ with the new multiplication
\begin{equation}\label{circ}
x\circ y=L_c^{-1}(xy).
\end{equation}
Then $(\mathbb{A},\circ)$ is a unital commutative associative algebra.
\end{proposition}

\begin{proof}
The algebra  $(\mathbb{A},\circ)$ is commutative and $c$ is a unit in $(\mathbb{A},\circ)$. Also, by \ref{ideal1}, $L_c^{-1}$ is $\mathbb{A}$-isomorphism and
$$
(x\circ y)\circ (z\circ w)=L_c^{-1}(L_c^{-1}(xy)\,L_c^{-1}(zw))=L_c^{-2}((xy)(zw))
$$
which by the nongeneracy of $L_c$ proves that $(\mathbb{A},\circ)$ is medial if and only if $\mathbb{A}$ is medial. By Proposition~\ref{pro:unital}, $(\mathbb{A},\circ)$ is  associative.
\end{proof}

\begin{remark}
Although relation ~\eqref{circ} establishes the existence of an associative algebra structure on the isotopy $(\mathbb{A},\circ)$, \textit{there is no canonical  way to reconstruct  the original algebra structure from}  $(\mathbb{A},\circ)$. It would be interesting to characterize all commutative algebras $\mathbb{A}$ whose associative  isotopes $(\mathbb{A},\circ)$ are  isomorphic.
%Even the simpler question about
\end{remark}

\subsection{The determinant on a medial algebra}
We have already seen in Examples~\ref{ex1} and \ref{ex2} that medial algebras in dimension 2 and 3 carry a multiplicative homomorphism.
As a corollary of Proposition~\ref{pro:ass} is the following result.
\begin{theorem}
\label{the:det}
Let $\mathbb{A}$ be a medial algebra, $c\in \Idm^{\times}(A)$. Then
$$\phi(x)=\frac1{\det L_c}\det L_x
$$
is $\mathbb{A}$-multiplicative, i.e.
$$
\phi(xy)=\phi(x)\phi(y), \qquad \forall x,y\in \mathbb{A}.
$$

\end{theorem}

\begin{proof}
Indeed, the algebra $(\mathbb{A},\circ)$ in Proposition~\ref{pro:ass} is associative. Let $M_x:y\to x\circ y$ denote the (left) multiplication operator in $(\mathbb{A},\circ)$. The multiplication operators in $(\mathbb{A},\circ)$ and $\mathbb{A}$ are related by $M_x=L_c^{-1}L_x$. By the associativity we have $M_{x\circ y}=M_xM_y$, hence by the multiplicative property of the determinant
\begin{equation}\label{ML}
\det M_{x\circ y}=\det (M_xM_y)=\det M_x \det M_y.
\end{equation}
Next, by \ref{ideal1} we have the $\mathbb{A}$-homomorphism property:
\begin{equation}\label{homom}
L_c^{-1}(xy)=L_c^{-1}(x)L_c^{-1}(y), \qquad x,y\in \mathbb{A}.
\end{equation}
With the latter relation in hand we obtain
$$
(L_{c}^{-1}z)x=(L_{c}^{-1}z)(L_{c}^{-1}L_cx)=L_{c}^{-1}(zL_cx)=
L_{c}^{-1}L_z L_cx
$$
implying
\begin{equation}\label{homom1}
L_{L_{c}^{-1}z}=L_{c}^{-1}L_z L_c.
\end{equation}
Setting $z=xy$ and applying the determinant we obtain
\begin{equation}\label{homomdet}
\det L_{L_{c}^{-1}(xy)}=\det L_{xy}.
\end{equation}
On the other hand, since $L_{c}^{-1}(xy)=x\circ y$ we get
\begin{equation}\label{homomdet}
\det L_{x\circ y}=\det L_{xy},
\end{equation}
which yields
$$
\det M_{x\circ y}=\det L_c^{-1}L_{x\circ y}=\det L_c^{-1}\det L_{xy}=\det L_c^{-1}L_{xy}=\det M_{xy}.
$$
Combining the latter with \eqref{ML} we finally obtain
$$
\det M_{xy}=\det M_x \det M_y
$$
therefore $\phi(x)=\det M_x=\det L_c^{-1}\det L_x$ is multiplicative in $\mathbb{A}$.
\end{proof}

\subsection{Further properties of medial algebras}
The following two propositions follow by straightforward examination of \eqref{magma}.

\begin{proposition}
\label{pro:product}
Let $\mathbb{A}$ and $\mathbb{B}$ be medial algebras. Then the direct product $\mathbb{A}\times\mathbb{B}$ with
\begin{align*}
(a,x)+(b,y)&=(a+b,x+y),\\
(a,x)\cdot(b,y)&=(ab,xy)
\end{align*}
is a medial algebra. The medial subalgebras $\mathbb{A}\times0$ and $0\times \mathbb{B}$ are ideals of $\mathbb{A}\times\mathbb{B}$.
\end{proposition}

\begin{proposition}
\label{pro:ideal}
Let $\mathbb{B}\subset \mathbb{A}$ be an ideal of a medial algebra. Then the factor algebra $\mathbb{A}/\mathbb{B}$ is medial too.
\end{proposition}

There is another useful  construction of medial algebras from a given one.

\begin{proposition}[Twisted doubling]\label{pro:epsilon}
Let $\mathbb{A}$  be a medial algebra and let $\zeta\in \{-1,1\}$. Define the new algebra $(\mathbb{A}\times \mathbb{A})_\zeta$ with multiplication
$$
(x,y)\circ (z,w)=(xz+yw, \zeta(xw+yz)).
$$
Then $(\mathbb{A}\times \mathbb{A})_\zeta$ is medial and $\mathbb{A}\cong \mathbb{A}\times0$ is a medial subalgebra of $(\mathbb{A}\times \mathbb{A})_\zeta$.
\end{proposition}

\begin{proof}
Setting $u=(x,y)$ and $v=(z,w)$, we obtain $u\bullet u=(x^2+y^2,2\zeta xy)$, hence
\begin{align*}
(u\bullet u)\bullet (v\bullet v)&=((x^2+y^2)(z^2+w^2)+4(xy)(zw),\, 2(x^2+y^2)zw+2(z^2+w^2)xy)\\
(u\bullet v)\bullet (u\bullet v)&=((xz+yw)^2+(xw+yz)^2,\, 2(xz+yw)(xw+yz)).
\end{align*}
By \eqref{magma}, the  right hand sides are equal, which implies \eqref{medial0}, therefore $(\mathbb{A}\times \mathbb{A})_\zeta$ is medial. Since $(x,0)\circ (z,0)=(xz, 0)$, $\mathbb{A}\times0$ is a medial subalgebra of $(\mathbb{A}\times \mathbb{A})_\zeta$.
\end{proof}

The twisted doubling is a particular case of the following general construction. Suppose that $\zeta$ is a root of one (not necessarily primitive) in $\Field$ of degree $d$, i.e. $\zeta^d=1$.
Let $\mathbb{A}[z]$ be the (nonassociative commutative) polynomial ring over a medial algebra $\mathbb{A}$. Obviously $\mathbb{A}[z]$  satisfies \eqref{magma}. Next, define on $\mathbb{A}[z]$  a new multiplication by
$$
p(z)\bullet q(z)\equiv p(\zeta z)q(\zeta z)\mod z^d-1.
$$
Thus obtained algebra  $\mathbb{A}^{\times d}_\zeta$  is a commutative nonassociative algebra of dimension
$$\dim\mathbb{A}^{\times d}_\zeta= d\dim \mathbb{A}.
$$
A less obvious property is the following

\begin{proposition}
\label{pro:twistp}
If $\mathbb{A}$ is medial so is $\mathbb{A}^{\times d}_\zeta$. The algebra $\mathbb{A}^{\times d}_1$ is unital with unit $e=1$.
\end{proposition}

\begin{proof}
Setting $w=\zeta^2 z$ we obtain
\begin{align*}
(p(z)\bullet q(z))\bullet (r(z)\bullet s(z))&\equiv (p(\zeta^2 z)q(\zeta^2 z))(r(\zeta^2 z)s(\zeta^2 z))\mod z^d-1\\
&\equiv (p(w)q(w))(r(w)s(w))\mod w^d-1\\
&\equiv (p(w)r(w))(q(w)s(w))\mod w^d-1\\
&\equiv (p(\zeta^2 z)r(\zeta^2 z))(q(\zeta^2 z)s(\zeta^2 z))\mod z^d-1\\
&=(p(z)\bullet r(z))\bullet (q(z)\bullet s(z)).
\end{align*}
as desired. Also, if $\zeta=1$ then $p(z)\bullet 1=p(\zeta z)=p(z)$, i.e. $1$ is the unit in $\mathbb{A}^{\times d}_1$.
\end{proof}

Comparing the above definitions reveals also that
$$
(\mathbb{A}\times \mathbb{A})_\zeta\cong \mathbb{A}^{\times 2}_\zeta.
$$
We shall see that the simplest case when $\mathbb{A}=\Field$ plays an essential role in the characterization of medial isospectral generic algebras given below.

\subsection{Some further examples of medial algebras}
Medial algebras appear in several different contexts.

%\subsection{Medial algebras and genetic algebras}

\begin{example}\label{ex4}
First note that there is  a nice relation between the squared magma identity \eqref{medial0} and \textit{genetic algebras}. Recall that a commutative algebra is called \textit{baric} if there exists a {nontrivial} algebra homomorphism  $\omega:\mathbb{A}\to \Field$, also called the weight function, which carries a certain probabilistic information. A baric algebra is called a \textit{Bernstein algebra} of order zero if it satisfies the  identity
\begin{equation}\label{Bern0}
x^2=\omega(x)x.
\end{equation}
The latter equation models a population with
random mating \cite{Glivenko36}, \cite{Bernstein42}, \cite{Etherington}.  Note that a  general commutative baric algebra $\mathbb{A}$ can be made up into a Bernstein algebra of order zero by introducing the multiplication
$$
x\circ y=\frac12(\omega(x)y+\omega(y)x).
$$
Indeed, since $\omega(x\circ y)=\omega(x)\omega(y)$, the algebra $(\mathbb{A},\omega)$ is baric and $x\circ x=\omega(x)x$, as desired.
Any Bernstein order zero algebra satisfies
$$
(xy)^2=\omega(xy)xy=(\omega(x)x)(\omega(y)y)=x^2y^2,
$$
which implies

\begin{corollary}
Any Bernstein  algebra $\mathbb{A}$ of order zero is medial.
\end{corollary}

On the other hand, the class of genetic algebras is too thin to cover all medial algebras. The algebra $\mathbb{A}_2$ from Example~\ref{ex1} is medial but not Bernstein. Indeed, if $\omega$ is a weight function on $\mathbb{A}_2$ then  we obtain from \eqref{Bern0} that $c_i^2=\omega(c_i)c_i$, therefore $\omega(c_i)=1$ for $i=1,2,3$. Since $c_i$ span $\mathbb{A}_2$, we deduce that $\omega\equiv 1$ on $\mathbb{A}_2$, i.e. $\omega$ is a trivial homomorphism, a contradiction.
In fact, as we know, $\mathbb{A}_2$ satisfies a 3rd order identity \eqref{B2}.
\end{example}

%\begin{example}
%\label{ex:baric14}
%Let us consider a train baric algebra satisfying
%\begin{equation}\label{baric1}
%x^3-(\gamma+1)\omega(x)x^2+\gamma \omega(x)^2x=0
%\end{equation}
%\end{example}

%\begin{example}\label{ex:rank3}
%Example~\ref{ex4} shows that any  rank two algebra is a medial algebra. Recall that a commutative algebra has rank $k$ if for any element $x$, all nonassociative powers of order less or equal $k$ are linearly dependent. The algebras of rank three has been thoroughly studied by Walcher in \cite{Walcher1}. Let
%\end{example}

\begin{example}\label{ex5}
Note that $\Field$ is a one-dimensional medial (associative) algebra over $\Field$. Following Proposition~\ref{pro:epsilon}, let us consider an algebra  $(\Field\times \Field)_\epsilon$. When $\epsilon=1$, the element $(1,0)$ is the algebra unit, hence by Proposition~\ref{pro:unital}, $(\Field\times \Field)_1$ is associative. Let us consider $\epsilon=-1$. Then  the idempotent equation in $(\Field\times \Field)_{-1}$ for $u=(x,y)$ is equivalent to the system
$$
\left\{
\begin{array}{cc}
x^2+y^2&=x\\
2xy +y^2&=0
\end{array}
\right.
$$
If $y=0$ then $c_0=(0,0)$ and $c_1=(1,0)$. If $y\ne0$ then $x=-\frac12$ and $y^2=-\frac34$. We have two cases:

\begin{itemize}
\item[(i)] If $t^2=-3$ is non-solvable in $\Field$ then $c=(1,0)$ is the only nonzero idempotent in $(\Field\times \Field)_{-1}$.
\item[(ii)]
If $t^2=-3$ is solvable in $\Field$ then one has exactly three nonzero idempotents $\{(1,0),(-\frac12,- \frac12t),(-\frac12, \frac12t)\}$, i.e. $(\Field\times \Field)_{-1}$ is generic.
\end{itemize}

Let us illustrate the above cases by the finite fileds $\Field =\mathbb{F}_5$ and $\Field =\mathbb{F}_7$. In the first case, $t^2\equiv -3\mod 5$ is non-solvable in $\mathbb{F}_5$, hence we have (i). If $\Field =\mathbb{F}_7$ then $t^2\equiv -3\equiv 4\mod 7$ is solvable in $\mathbb{F}_5$, $t=\pm2$. Since $-1/2\equiv 3 \mod 7$, the corresponding idempotents are $c_2=(3,1)$ and $c_3=(3,6)$. Furthermore,
\begin{align*}
c_1c_2&=(1,0)\circ (3,1)=(3,6)=c_3\\
c_2c_3&=(3,1)\circ (3,6)=(1,0)=c_1\\
c_3c_1&=(3,6)\circ (1,0)=(3,1)=c_2,
\end{align*}
and $c_1+c_2+c_3=0$ imply \eqref{I2} and \eqref{cicjck}. This shows that
$$
\mathbb{A}_2(\mathbb{F}_7)\cong (\mathbb{F}_7\times \mathbb{F}_7)_{-1}.
$$
Note, however, that  $\mathbb{A}_2(\mathbb{F}_5)\not\cong (\mathbb{F}_5\times \mathbb{F}_5)_{-1}$ for the reasons indicated above.

In general, it is well known that $-3$ is a quadratic residue modulo $p$ if and only if $p=3k+1$. Therefore, if $\Field=\mathbb{F}_p$ then $\mathbb{A}_2(\mathbb{F}_p)\cong (\mathbb{F}_p\times \mathbb{F}_p)_{-1}$ if and only if $p=3k+1$. Note, however, that if $p\ne 3k+1$ then  $\mathbb{A}_2(\mathbb{F}_p)$ is a nontrivial isospectral algebra, but it is not isomorphic to $(\mathbb{F}_p\times \mathbb{F}_p)_{-1}$.

Finally, note that
$$
\mathbb{A}_2(\mathbb{C})\cong (\mathbb{C}\times \mathbb{C})_{-1}.
$$
\end{example}

\section{Idempotents in medial algebras}

\subsection{The zero Peirce eigenspace}
The following example shows that there exist medial algebras with non-empty set $\Idm(\mathbb{A})\setminus \Idm^\times(\mathbb{A})$.

\begin{example}
Let $\mathbb{A}$ be an arbitrary medial algebra over $\Field$. Then the direct product  $\mathbb{A}\times \Field$ is a medial algebra by Proposition~\ref{pro:product}. Clearly, $c_0=(0,1)$ is an idempotent in $\mathbb{A}\times \Field$.
We also know by Proposition~\ref{pro:product} that $\mathbb{A}\times 0$ and $0\times \Field$ are ideals. Then $\mathbb{A}\times 0=\ker L_{c_0}$.
\qed
\end{example}

Let $\mathbb{A}$ be a medial algebra and let $f:\mathbb{A}\to \mathbb{A}$ be an algebra endomorphism. Then $\ker f$ is a (medial algebra) ideal in $\mathbb{A}$ and the factor $\mathbb{A}/\ker f$  obviously satisfies the medial magma identity \eqref{magma}. We have the direct product of medial algebras
$$
\mathbb{A}\cong\ker f\times (\mathbb{A}/\ker f).
$$
%Given an arbitrary idempotent $c\in\mathbb{A}$, the class $c+\ker f\in \mathbb{A}/\ker f$ is  also an idempotent in $\mathbb{A}/\ker f$. Conversely, an element $x+\ker f\in \Idm(\mathbb{A}/\ker f)$ if and only if $f(x)\in \ker f$.

%\begin{proposition}
%Let $\mathbb{A}$ be an algebra and let $f:\mathbb{A}\to \mathbb{A}$ be an algebra endomorphism. Then
%$$\Idm(\mathbb{A}/\ker f)=f^{-1}(\Idm(\mathbb{A}))/\ker f.
%$$
%\end{proposition}
%
%\begin{proof}
%If $c\in \Idm(\mathbb{A})$ then $c+\ker f\in \Idm(\mathbb{A}/\ker f)$. Conversely, let $x+\ker f\in \Idm(\mathbb{A}/\ker f)$. Since the latter idempotent in nonzero, we have $f(x)\ne0$. Moreover, $x^2-x\in \ker f$, hence $f(x)^2=f(x)$. It follows that $f(x)\in \Idm(A)$. Therefore, $x+\ker f\in f^{-1}(\Idm(\mathbb{A}))+\ker f$
%\end{proof}

Now, if $c$ is an idempotent in a medial algebra $A$ then by \ref{ideal1} in Proposition~\ref{pro:endom}, $L_c$ is an algebra homomorphism. Hence the zero Peirce eigenspace
$$
\mathbb{A}_c(0)=\ker L_c,
$$
is a (left=right) ideal of $\mathbb{A}$. Therefore,
%By Proposition~\ref{pro:Ac0},
$$
\mathbb{A}\cong \mathbb{A}_c(0)\oplus \mathbb{A}/\mathbb{A}_c(0),
$$
where $\mathbb{A}/\mathbb{A}_c(0)$ is medial (a trivial verification). Moreover, the class $c'=c+\mathbb{A}_c(0)$ is obviously an idempotent in $\mathbb{A}/\mathbb{A}_c(0)$. We have
\begin{align*}
\mathbb{A}_{c'}(0)&=\{x+\mathbb{A}_c(0): (c+\mathbb{A}_c(0))(x+\mathbb{A}_c(0))=\mathbb{A}_c(0)\}\\
&=\{x+\mathbb{A}_c(0): cx\in \mathbb{A}_c(0)\}\\
&=\{x+\mathbb{A}_c(0): L_c^2x=0\}\\
&\equiv \ker L_c^2\mod \mathbb{A}_c(0)
\end{align*}

\begin{proposition}\label{pro:Ac0}
The Peirce subspace $\mathbb{A}_c(0)$ is an ideal in $\mathbb{A}$ and the factor $\mathbb{A}'=\mathbb{A}/\mathbb{A}_c(0)$ is a medial algebra. If additionally $c$ is semisimple then $\mathbb{A}'_{c'}(0)=0$, where $c'$ is the class of $c$ in $\mathbb{A}'$.
\end{proposition}

\begin{proof}
If  $x\in\mathbb{A}$ and $y\in \mathbb{A}_c(0)$ then applying \eqref{medial0a} yields $$
c(xy)=(cc)(xy)=(cx)(cy)=0,
$$
thus, $xy\in \mathbb{A}_c(0)$. This implies that  $A\,\mathbb{A}_c(0)\subset \mathbb{A}_c(0)$, i.e. $\mathbb{A}_c(0)$ is an ideal of $\mathbb{A}$. Next, if $u_i=x_i+y_i$ with $y_i\in \mathbb{A}_c(0)$, then
$$
u_i^2=x_i^2+z_i, \qquad \text{where }z_i=2x_iy_i+y_i^2\in \mathbb{A}_c(0),
$$
which readily yields
$$
u_1^2u_2^2=x_1^2x_2^2+z_1x_2^2+z_2x_1^2+z_1z_2=x_1^2x_2^2+z_3,
$$
where $z_3\in \mathbb{A}_c(0)$. Similarly,
$$
(u_1u_2)^2=(x_1x_2+x_1y_2+y_1x_2+y_1y_2)^2=(x_1x_2+z_4)^2=(x_1x_2)^2+z_5,
$$
where $z_4,z_5\in \mathbb{A}_c(0)$. This shows that $u_1^2u_2^2=(u_1u_2)^2$ in $\mathbb{A}'=\mathbb{A}/\mathbb{A}_c(0)$, hence by Proposition~\ref{pro:med} the  factor algebra $\mathbb{A}'$  is also medial.

Further, since $L_c$ annihilate $\mathbb{A}_c(0)$,
$$
(c+\mathbb{A}_c(0))^2=c^2+\mathbb{A}_c(0)^2=c+\mathbb{A}_c(0),
$$
therefore the class $c'=c+\mathbb{A}_c(0)$ is an idempotent in $\mathbb{A}'$. Since $\mathbb{A}_c(0)$ is an ideal we have
\begin{align*}
\mathbb{A}_{c'}(0)&=\{x+\mathbb{A}_c(0): (c+\mathbb{A}_c(0))(x+\mathbb{A}_c(0))=0'=\mathbb{A}_c(0)\}\\
&=\{x+\mathbb{A}_c(0): cx\in \mathbb{A}_c(0)\}\\
&=\{x+\mathbb{A}_c(0): L_c^2x=0\}
\end{align*}
Suppose additionally that  $c$ is semisimple. Then we can decompose $x$ into a direct sum $x=x_0+\sum_{\lambda\ne0}x_\lambda$, where $x_\lambda\in \mathbb{A}_c(\lambda)$. This yields
$$
0=L_c^2x=\sum_{\lambda\ne0}\lambda^2 x_\lambda,
$$
hence $x_\lambda=0$ for all $\lambda\ne0$, i.e. $x=x_0\in \mathbb{A}_c(0)$. This implies $x'=0'$, hence the proposition follows.
\end{proof}

\begin{definition}
A medial algebra $\mathbb{A}$ is said to be \textit{reduced}  if $\Idm^\times(\mathbb{A})=\Idm(\mathbb{A})$.
\end{definition}

\begin{remark}
Note that $\dim \mathbb{A}/\mathbb{A}_c(0)=\dim \mathbb{A}-\dim \mathbb{A}_c(0)<\dim \mathbb{A}$ if $\mathbb{A}_c(0)$ is nontrivial. Hence, continuing in an obvious way  factorizing out of $\mathbb{A}$ nontrivial Peirce zero subspaces, one arrives in finitely many steps to a reduced medial algebra $\tilde{A}$. We shall  consider reduced medial algebras in more detail below.
\end{remark}

\subsection{The Peirce structure of reduced medial algebras}

Throughout this section we assume, unless stated otherwise, that $\mathbb{A}$ is a reduced medial algebra of dimension $\ge 2$. Since $\ker L_c=0$ for any idempotent $c$, $L_c$ is an algebra automorphism. Given $\lambda\in \Field$ we denote
$$
\mathbb{A}_c(\lambda)=\ker (L_c-\lambda \mathbf{1})
$$

\begin{proposition}
\label{pro:spect}
Let $c_1,c_2\in \Idm(\mathbb{A})$. Then $c_1c_2\in \Idm(\mathbb{A})$ and
\begin{align}\label{mulPeirce}
\mathbb{A}_{c_1}(\lambda_1)\mathbb{A}_{c_2}(\lambda_2)&\subset \mathbb{A}_{c_1c_2}(\lambda_1\lambda_2),
\\
\label{mulPeirceC}
\mathbb{A}_{c_1}(\lambda_1)\mathbb{A}_{c_1}(\lambda_2)&\subset \mathbb{A}_{c_1}(\lambda_1\lambda_2).
\end{align}
In particular,
\begin{equation}\label{mulPeirce2}
\sigma(c_1)\cup \sigma(c_2)\subset \sigma(c_1c_2)
\end{equation}
and
\begin{equation}\label{mulPeirce3}
\dim \mathbb{A}_{c_j}(\lambda_j)\le \dim \mathbb{A}_{c_ic_j}(\lambda_j), \qquad \forall \lambda_i\in \sigma(c_i), i=1,2.
\end{equation}
\end{proposition}

\begin{proof}
If some of $\mathbb{A}_{c_1}(\lambda_1)$ and $\mathbb{A}_{c_2}(\lambda_2)$ is zero, \eqref{mulPeirce}, \eqref{mulPeirceC} are trivial. Suppose that both  $\mathbb{A}_{c_1}(\lambda_1)$ and $\mathbb{A}_{c_2}(\lambda_2)$ are nonzero and let
$c_ix_i=\lambda_i x_i$, $x_i\ne 0$, $i=1,2$. Then $c_3=c_1c_2$ is  an idempotent, and, since $\mathbb{A}$ is reduced, we have $c_3\ne0$. Using \eqref{magma} obtain
$$
c_3(x_1x_2)=(c_1c_2)(x_1x_2)=(c_1x_1)(c_2x_2)=\lambda_1\lambda_2  x_1x_2.
$$
Thus $x_1x_2\in \mathbb{A}_{c_3}(\lambda_1\lambda_2)$. This implies \eqref{mulPeirce}. Setting $c_2=c_1$ we arrive at \eqref{mulPeirceC}. Finally, since $x_2$ is a $\lambda_2$-eigenvector of $c_2$,
$$
\lambda_2c_1x_2=c_1\lambda_2x_2=c_1(c_2x_2)=(c_1c_1)(c_2x_2)=
(c_1c_2)(c_1x_2)=c_3(c_1x_2).
$$
Since $\mathbb{A}$ is reduced, we have $c_1x_2\ne0$. Therefore $c_1x_2$ is a $\lambda_2$-eigenvector of $c_3$. This proves \eqref{mulPeirce2}. Furthermore, since $L_{c_1}$ is non-degenerated linear endomorphism, we also deduce that
\begin{equation}\label{mulPeirce1}
0\ne L_{c_1}(\mathbb{A}_{c_2}(\lambda_2))\subset \mathbb{A}_{c_1c_2}(\lambda_2).
\end{equation}
Since $\mathbb{A}$ is reduced, \eqref{mulPeirce1} implies \eqref{mulPeirce3}.
\end{proof}

Note that the product in the left hand side maybe zero such that the inclusion in \eqref{mulPeirceC} is trivial in that case.

\begin{proposition}\label{pro:semi}
If  $c_1$ is a semi-simple idempotent then for any idempotent $c_2\in \Idm(\mathbb{A})$, the idempotent $c_1c_2$ is also semi-simple and  $$\spec (c_1)=\spec (c_1c_2).
$$
\end{proposition}

\begin{proof}
By the assumption we have the following direct sum decomposition:
$$
A=\bigoplus_{\lambda_i\in \sigma (c_1)} \mathbb{A}_{c_1}(\lambda_i),
$$
where all $\lambda_i$ are  distinct. Since $\mathbb{A}$ is reduced, $L_{c_2}$ is a nondegenerate endomorphism of $\mathbb{A}$, hence
$$
A=\bigoplus_{\lambda_i\in \sigma (c_1)} L_{c_2}(\mathbb{A}_{c_1}(\lambda_i)).
$$
Comparing this with \eqref{mulPeirce1} and \eqref{mulPeirce3} imply for dimensional reasons that
$$
L_{c_2}(\mathbb{A}_{c_1}(\lambda_i))=\mathbb{A}_{c_1c_2}(\lambda_i),
$$
hence
$$
A=\bigoplus_{\lambda_i\in \sigma (c_1)} \mathbb{A}_{c_1c_2}(\lambda_i).
$$
yields that $c_1c_2$ is a semi-simple idempotent. Since $\dim \mathbb{A}_{c_1c_2}(\lambda_i)=\dim \mathbb{A}_{c_1}(\lambda_i)$, the spectra of $c_1c_2$ and $c_1$ are identical.
\end{proof}

\begin{corollary}\label{cor:semi}
If  $c_1$ and $c_2$ are semi-simple idempotents then $\spec (c_1)=\spec (c_2)$.
\end{corollary}

\begin{proof}
It follows from Proposition~\ref{pro:semi} that the spectra of $c_1$,  $c_2$ and $c_1c_2$ are identical, that implies the desired conclusion.
\end{proof}

\begin{corollary}\label{cor:isosp}
If all idempotents of a reduced medial algebra $\mathbb{A}$ are semi-simple then $\mathbb{A}$ is isospectral.
\end{corollary}

\begin{corollary}\label{cor:isosp2}
If a reduced medial algebra $\mathbb{A}$ is generic and all its idempotents are semi-simple then  $\mathbb{A}$ is isospectral and its spectrum consists of the primitive roots of unity of degree $\dim \mathbb{A}$:
$$\spec (c)=\{\epsilon_n^k: 1\le k\le n=\dim A\}, \qquad \forall c\in \Idm(\mathbb{A}).
$$
\end{corollary}

\begin{proof}
Follows from the syzygy relation in Theorem~\ref{th4}.
\end{proof}

\section{Isospectral algebras}\label{sec:iso}

\subsection{Basic properties}
\begin{definition}
A commutative algebra is called \textit{isospectral} if all its nonzero idempotents have the same spectrum.
\end{definition}

We point out that we require in the definition that the spectra of each two idempotents coincide as a multi-set, i.e. counting the multiplicities of all eigenvalues.

In the generic case, one a priori knows that there exists exactly the maximal number $2^n-1$ idempotents, where $n=\dim \mathbb{A}$, so that the  condition `all idempotents' has a clear sense. In the nongeneric case, however, there may exist either nontrivial 2-nilpotent elements or infinitely many idempotents, and thus, containing the exceptional value $\frac12$ in the spectrum, see \cite{KrTk18a}, \cite{Tk18e}. The constancy of the spectrum still has an exact meaning, but the situation in the nongeneric case will be completely different. For example, if there is only few distinct idempotents (but many nilpotents)  then the constancy spectrum condition is not so much restrictive. On the other hand, there are nonisospectral generic algebras, for example, Hsiang algebras (see \cite{Tk18e} for the definition and basic properties of Hsiang algebras). It would be interesting to characterize the Peirce spectrum  of nonisospectral generic algebras.
%So far, we have some further examples of non-generic algebras, see the construction and discussion in section~\ref{sec:app} below.

In this paper we shall primarily concern with the generic case. Then the most natural choice of the ground field is the complex numbers $\mathbb{C}$ or an algebraically closed field $\Field$. Then by Corollary~5.2 in \cite{KrTk18a} the spectrum of a generic algebra consists of the primitive roots of unity of order $n=\dim \mathbb{A}$. More precisely, applying the syzygies constructed in \cite{KrTk18a}, one derives a more precise description.

\begin{proposition}
\label{pro:isosp1}
Let $\mathbb{A}$ be a generic  isospectral algebra over an algebraically closed field $\Field$. Let $n=\dim \mathbb{A}$. Then for any $c\in\Idm(\mathbb{A})$
\begin{equation}\label{circular}
\spec (c)=\{\epsilon_n^k: 1\le k\le n\}
\end{equation}
and the following Peirce decomposition holds:
$$
A=\bigoplus_{k=0}^{n-1}\mathbb{A}_c(\epsilon_n^k)=\bigoplus_{k=0}^{n-1}\Span{w_k},
$$
where $\{w_0,w_1,\ldots, w_{n-1}\}$ is an eigenbasis of $L_c$: $cw_k=\epsilon_n^k w_k$, where the corresponding Peirce eigenspaces are one-dimensional:
\begin{equation}\label{w1}
\mathbb{A}_c(\epsilon_n^k)=\Span{w_k}=\Field w_k.
\end{equation}
%In this notation, one can additionally assume that  $w_0=c$.
In particular, the algebra $\mathbb{A}$ does not contain $2$-nilpotents and
\begin{equation}\label{eqiso}
L_c^n=1, \qquad \forall c\in\Idm(\mathbb{A}).
\end{equation}
Furthermore, for any homogeneous polynomial map $H:\mathbb{A}\to \mathbb{A}$ of degree $1\le m\le n-1$,
\begin{align}
\sum_{c\in \Idm (\mathbb{A})}H(c)&=0.\label{syzy1}
\end{align}
In particular,
\begin{align}
\sum_{c\in \Idm(\mathbb{A})}c&=0.\label{syzy2}
\end{align}
\end{proposition}

\begin{proof}
By the assumptions, there exists exactly $2^n-1$ nonzero idempotents in $\mathbb{A}$ and the characteristic polynomials $\chi_c(t)=P(t)$ for some fixed polynomial $P$ of degree $n$. Applying  Theorem~\ref{th:gener}, we obtain from \eqref{polynom} that $(2^n-1)P(t)=2^n(1-t^n)P(\frac12)$ which immediately implies $P(t)=t^n-1$. Since the spectrum of each nonzero idempotent $c$ in $\mathbb{A}$ is the multiset of all roots of $P(t)$, we conclude that $L_c$ has exactly $n$ distinct eigenvalues $\{1,\epsilon_n, \ldots, \epsilon_n^{n-1}\}$, where $\epsilon_n$ is an $n$th primitive root of unity in $\Field$. In particular, for dimension reasons, we conclude that for each eigenvalue $\epsilon_n^{k}$ its eigenspace is one-dimensional. Furthermore, \eqref{syzy1} and \eqref{syzy2} follows from \eqref{EuJa4} and \eqref{EuJa2} respectively.
\end{proof}

%\begin{corollary}
%\label{cor:Peirceiso}
%Let $\mathbb{A}$ be a isospectral generic algebra. Then for any $c\in\Idm(\mathbb{A})$, the following Peirce decomposition holds:
%$$
%A=\bigoplus_{k=0}^{n-1}\mathbb{A}_c(\epsilon_n^k)=\bigoplus_{k=0}^{n-1}\Span{w_k},
%$$
%where $\{w_0,w_1,\ldots, w_{n-1}\}$ is an eigenbasis of $L_c$: $cw_k=\epsilon_n^k w_k$. Furthermore, the corresponding Peirce eigenspaces are one-dimensional:
%\begin{equation}\label{w1}
%\mathbb{A}_c(\epsilon_n^k)=\Span{w_k}=\Field w_k.
%\end{equation}
%In this notation, one can additionally assume that  $w_0=c$.
%\end{corollary}
%
%\begin{proof}
%Let  $c\in\Idm(\mathbb{A})$. By Proposition~\ref{pro:isosp1}, the spectrum of $L_c$ is cyclotomic,
%\end{proof}

Note that in Proposition~\ref{pro:isosp1} one can additionally assume that  $w_0=c$. This motivates

\begin{definition}
Let $c\in \Idm(\mathbb{A})$. A basis  $\{w_i\}_{0\le i\le n-1}$ is called an \textit{eigenbasis associated with $c$} if $w_0=c$ and
$$
cw_i=\epsilon_n^i w_i \quad \text{for all $1\le i\le n-1$}.
$$
\end{definition}

Let  $y\in\mathbb{A}$ be a fixed element and $1\le m\le n-1$ be an integer. Applying $H(x)=L_x^my$ to \eqref{syzy1} yields
\begin{equation}\label{Lc0}
\sum_{c\in \Idm (\mathbb{A})}L_c^my=0.
\end{equation}

\begin{corollary}
\label{cor:trace}
Let $\mathbb{A}$ be an isospectral generic algebra.  For any $c\in \Idm(\mathbb{A})$ and integer $k$, there exists a linear form $\theta_{c,k}(y):\mathbb{A}\to \Field$ such that
\begin{equation}\label{zeta1}
y+\epsilon_n^{-k}L_cy+\epsilon_n^{-2k}L_c^{2}y+\ldots +\epsilon_n^{-(n-1)k}L_c^{n-1}y=\theta_{c,k}(y)w_k,
\end{equation}
and
$$
\theta_{c,k}(w_j)=
\left\{
\begin{array}{cl}
n& \text{if $k=j$}\\
0& \text{if $k\ne j$}.
\end{array}
\right.
$$
In particular, if $y\ne 0$ then $y,L_cy,L_c^{2}y,\ldots,L_c^{n-1}y$ is a basis of $\mathbb{A}$ as a vector space and
\begin{equation}\label{zeta10}
y+L_cy+L_c^{2}y+\ldots +L_c^{n-1}y=\theta_{c,0}(y)c.
\end{equation}

\end{corollary}

\begin{proof}
Let $z$ denote the left hand side of \eqref{zeta1}. Applying $L_c$ to $z$ and rearranging we obtain
\begin{align*}
L_cz&=L_cy+\epsilon_n^{-k}L^2_cy+\epsilon_n^{-2k}L_c^{2}y+\ldots +\epsilon_n^{-(n-1)k}L_c^{n}y\\
&=\epsilon_n^{k}L_c^{n}y+L_cy+\epsilon_n^{-k}L^2_cy+\ldots +\epsilon_n^{-(n-2)k}L_c^{n-1}y\\
&=\epsilon_n^{k}(y+\epsilon_n^{-k}L^2_cy+\epsilon_n^{-2k}L_c^{2}y+\ldots +\epsilon_n^{-(n-1)k}L_c^{n-1}y)\\
&=\epsilon_n^kz,
\end{align*}
hence $z\in \mathbb{A}_c(\epsilon_n^k)$. By Proposition~\ref{pro:isosp1}, $\mathbb{A}_c(\epsilon_n^k)=\Span{w_k}$, hence $z=\lambda w_k$, for some complex number $\lambda$. Clearly, $\theta:y\to \lambda$ is a linear form, and \eqref{zeta1} follows. Substituting $y=w_j$ in \eqref{zeta1} yields
\begin{align*}
\theta_{c,k}(w_j)w_k&=(1+\epsilon_n^{j-k}+\epsilon_n^{2(j-k)}+\ldots +\epsilon_n^{(n-1)(j-k)})w_j\\
&=
w_j\left\{
\begin{array}{cl}
n& \text{if $k=j$}\\
\frac{\epsilon_n^{n(j-k)}-1}{\epsilon_n^{j-k}-1}& \text{if $k\ne j$}.
\end{array}
\right.
\end{align*}
Since $\epsilon_n^{n(j-k)}=1$, the sum vanishes whenever $j\ne k$, thus the claim follows. Also, if we consider \eqref{zeta1} as a system of $n$ equations (for $0\le k\le n-1$) with respect to $y,L_cy,L_c^{2}y,\ldots,L_c^{n-1}y$ then its determinant is of Vandermond type and it is obviously different of zero. Since the system $\{w_0,w_1,\ldots,w_{n-1}\}$ is a basis in $A$, so is also $y,L_cy,L_c^{2}y,\ldots,L_c^{n-1}y$.
\end{proof}

Summing up the identities \eqref{zeta10} for all $c\in \Idm(\mathbb{A})$, we obtain using \eqref{Lc0}  that

\begin{corollary}
\label{cor:trace1}
For any $y\in\mathbb{A}$
\begin{equation}\label{zeta2}
y=\frac{1}{(2^n-1)}\sum_{c\in \Idm(\mathbb{A})}\theta_{c,0}(y)c.
\end{equation}
In particular, $\Idm(\mathbb{A})$ spans $\mathbb{A}$, i.e. any isospectral generic algebra is generated by idempotents (even on the level of the vector space).
\end{corollary}

\begin{remark}
 Combining the latter with the property that any idempotent has the same spectrum and fusion rules, one concludes that an isospectral algebra is \textit{axial} in the sense of definitions in \cite{HRS15b},  \cite{Rehren17}, \cite{KMcSh2022}, \cite{McSh2022}. The axes in the present context are all idempotents of $\mathbb{A}$.
\end{remark}

%Some of the examples discussed in the introduction are metrized algebras. More precisely, an algebra  $\mathbb{A}$ is called metrized, if it carries a symmetric bilinear form $b:\mathbb{A}\times A\to \Field$ which is invariant, i.e.
%\begin{equation}\label{invarbili}
%b(xy,z)=b(x,yz), \qquad \forall x,y,z\in\mathbb{A}.
%\end{equation}
%
%\begin{proposition}
%\label{pro:invariant}
%If $\dim \mathbb{A}\ge 3$ is medial then it is not metrized.\marginpar{FINISH!}
%\end{proposition}
%
%\begin{proof}
%Let $c\in \Idm(\mathbb{A})$ and let $\{w_i\}_{0\le i\le n-1}$ be an eigenbasis associated with $c$. Then
%$$
%b(w_i,w_j)=\frac{1}{\epsilon_n^i}b(cw_i,w_j)=\frac{1}{\epsilon_n^i}b(w_i,cw_j)=
%\epsilon_n^{j-i}b(w_i,w_j).
%$$
%This proves that $b(w_i,w_j)=0$ whenever $j\ne i$. Next,
%$$
%b(w_i,w_i)
%$$
%\end{proof}

\subsection{The generic trace and the generic determinant}
Let $c\in \Idm(\mathbb{A})$ and let $\{w_i\}_{0\le i\le n-1}$ be an eigenbasis associated with $c$. Define for any $x=\sum_{i=0}^{n-1}a_iw_i$ the \textit{generic trace}
$$
T(x):=\sum_{i=0}^{n-1}a_i.
$$
We have
$$
L_c^kx=\sum_{i=0}^{n-1}a_iL_c^kw_i=
\sum_{i=0}^{n-1}\epsilon_n^{ik} a_iw_i,
$$
hence
$$
T(L_c^kx)=\sum_{i=0}^{n-1}\epsilon_n^{ik}a_i.
$$
Let us define
\begin{equation}\label{delta}
\delta(x):=\prod_{k=0}^{n-1}T(L_c^kx)=\det X(x),
\end{equation}
where $X$ is the circulant matrix of $x$:
$$
X_{i,j}=x_{\bar j-\bar i+1}, \qquad \bar i, \bar j\in \mathbb{Z}_n.
$$
The product will be called the \textit{generic determinant} of $x$ in $\mathbb{A}$.

%Given $x=\sum_{i=0}^{n-1}a_iw_i$ and $y=\sum_{j=0}^{n-1}b_jw_j$, we also define a symmetric bilinear form
%$$
%B(x,y)=\sum_{i,j=0}^{n-1}a_i b_j,
%$$
%%where $z\to \bar z$ is the complex conjugation.
%Then $L_cx=\sum_{i=0}^{n-1}a_i\epsilon_n^iw_i$, hence
%$$
%B(L_cx,y)=\sum_{i,j=0}^{n-1}\epsilon_n^ia_i b_j=
%\sum_{i,j=0}^{n-1}a_i\epsilon_n^i b_j=B(x,L_cy).
%$$
%This proves that $L_c$ is self-adjoint with respect to $B$.

\subsection{The generic determinant}
\begin{theorem}
Let $\mathbb{A}$ be a isospectral generic algebra. Then it satisfies the identity
\begin{equation}\label{identity}
L_x^ny=\delta(x)y,
\end{equation}
where $\delta(x)$ is the generic determinant of $x$.
\end{theorem}

%\subsection{Isospectral algebras and medial algebras}
%By Corollary~\ref{cor:isosp}, any semi-simple  reduced medial algebra is necessarily  isospectral. The converse property also holds true for all known examples of isospectral generic algebras. We conjecture that any isospectral algebra is medial.
%
%In this section we study generic algebras which are isospectral and medial simultaneously.  We have the following key observation.
%
%\begin{theorem}
%If $\mathbb{A}$ is a medial generic  isospectral  algebra then for any idempotent $c\in \Idm(\mathbb{A})$ there exists an eigenbasis $\{w_i\}_{0\le i\le n-1}$, $n=\dim \mathbb{A}$, and $cw_i=\epsilon_n^i w_i$, such that
%\begin{equation}\label{eigen1}
%w_iw_j=w_{i+j}.
%\end{equation}
%\end{theorem}
%
%We divide the proof in the several steps.
%
%
%
%
%
%
%
%
%
%\begin{corollary}
%If $\mathbb{A}$ is a medial generic  isospectral  algebra then
%\begin{equation}\label{deter}
%L_x^n=\Delta(x)1, \qquad \forall x\in\mathbb{A}.
%\end{equation}
%\end{corollary}
%
%\begin{proof}
%
%\end{proof}

\section{Medial isospectral algebras}\label{sec:med-iso}

Throughout this section, $\mathbb{A}$ denotes a medial generic  isospectral algebra over an algebraically closed subfield $\Field$ of  $\mathbb{C}$.
Let $N=2^n-1$ and let $\{c_i\}_{1\le i\le N}$ denote the set the distinct nonzero idempotents. The main result of this section is

\begin{theorem}
\label{the:uniq}
If two medial generic  isospectral algebras over an algebraically closed field $\Field$ have the same dimension  they are isomorphic.
\end{theorem}

We begin with some elementary observations.

\begin{proposition}\label{pro:obser}
Let $\mathbb{A}$ be a medial generic  isospectral algebra. Then

\begin{enumerate}[label=(\alph*)]

\item\label{ax1}
Any idempotent $c\in \Idm(\mathbb{A})$ is semi-simple.
\item\label{ax2}
For any two idempotents $c_i,c_j\in \Idm(\mathbb{A})$ the product $c_ic_j\in \Idm(\mathbb{A})$.
\item\label{ax3}
For any $0\ne x\in\mathbb{A}$, $x^2\ne0$.
\item\label{ax4}
If $c\in \Idm(\mathbb{A})$ and $\mathbb{A}_c(\epsilon_n^i)=\Span{w_i}$ then
\begin{equation}\label{Peircerules}
w_kw_j\in \mathbb{A}_c(\epsilon_n^{k+j})= \Span{w_{j+k}}, \qquad \forall k,j.
\end{equation}
\end{enumerate}
\end{proposition}

\begin{proof}
If $c\in \Idm(\mathbb{A})$ then $L_c^n=\mathbf{1}$ has only simple roots. Since the field is algebraically closed, for any eigenvalue $\epsilon_n^k$, $0\le k\le n-1$, there exists an eigenvector which implies for dimension reasons that each eigenspace $\mathbb{A}_c(\epsilon_n^k)$ has dimension 1. This implies that $\mathbb{A}$ decomposes into the direct sum $\bigoplus_{1\le i\le n-1}\mathbb{A}_c(\epsilon_n^k)$, hence \ref{ax1} follows.

If $c_ic_j= 0$ then $c_j$ is an eigenvector of $c_i$ with eigenvalue $0$, a contradiction with $L_{c_i}^n=1$ follows. The fact that $c_ic_j$ is an idempotent follows from Proposition~\ref{pro:idem}. This proves \ref{ax2}

Since $\mathbb{A}$ is generic, it does not contain 2-nilpotent elements (see Remark~3.2 in \cite{KrTk18a}), hence \ref{ax3} follows. Finally,  by \eqref{mulPeirce} and \eqref{w1} we have
$$
w_{j}w_{k}\subset \mathbb{A}_{c}(\epsilon_n^j)\mathbb{A}_{c}(\epsilon_n^j)\subset \mathbb{A}_{c}(\epsilon_n^{j+k})= \Span{w_{j+k}}.
$$
as desired.
\end{proof}

To proceed we need to show that the product in \eqref{Peircerules} never vanishes. First let us recall some standard definitions.

A nonassociative monomial $z^{\alpha }$ is an element of the commutative multiplicative groupoid generated by a single element $z$. The total degree  $\deg z^\alpha$ is the number of elements of $z$ used in the word $z^\alpha$, counting multiplicities. For example, $\deg z^2z^2=4$. The simplest nonassociative monomials are the principal powers, defined by induction as
\begin{equation}\label{princip}
z^1=z,\quad z^{n+1}=zz^n, \quad n\ge1.
\end{equation}
Then $\deg z^m=m$.

\begin{proposition}\label{pro:w}
Let $\mathbb{A}$ be a medial generic  isospectral  algebra and let $\mathbb{A}_c(\epsilon_n)=\Span{w_1}$. Then
\begin{enumerate}[label=(\roman*)]
\item\label{prow1}
$w_k=a_kw_1^m$, for some $0\ne a_k\in \Field$;
\item\label{prow2}
$w_kw_j\ne0$ for any $k,j$;
\item\label{prow3}
one can choose the Peirce eigenbasis such that $w_m=w_1^m$ for all $0\le m\le n-1$, where $w_1^0=c$.
\end{enumerate}
\end{proposition}

\begin{proof}
Let $w_k$ and $w_j$ be eigenvectors of $c$ such that $w_kw_j=0$. Then for any $x\in\mathbb{A}$ there holds $(xw_k)w_j=w_k(xw_j)=0$. Indeed,
$$
(w_kx)w_j=\frac{1}{\epsilon_n^j}(w_kx)(w_jc)= \frac{1}{\epsilon_n^j}(w_kw_j)(xc)=0.
$$
We refer to this as the \textit{cancelation property}.

We claim that any principal power of $w_j$ is nonzero. Indeed,  $w_j^2\ne0$ by \ref{ax3} in Proposition~\ref{pro:obser}. Let $w_j^m=0$ for some principal power $m$ and let $m$ be the smallest such number. Then $m\ge 3$. Using the medial algebra relation,
\begin{align*}
w_1^{m-1}w_j^{m-1}&=\frac{1}{\epsilon_n^{m-1}}(w_j^{m-1}c)(w_jw_j^{m-2})= \frac{1}{\epsilon_n^{m-1}}(w_jw_j^{m-1})(w_j^{m-2}c)\\
&=\frac{1}{\epsilon_n^{m-1}}(w_j^{m})(w_j^{m-2}c)=0,
\end{align*}
a contradiction with $w_j^{m-1}\ne0$ and \ref{ax3} follows. This proves our claim.

By \eqref{Peircerules}, $w_kw_j\subset \mathbb{A}_c(\epsilon_n^{k+j})$. Moreover, since the latter eigenspace is one-dimensional, $\Span{w_kw_j}= \mathbb{A}_c(\epsilon_n^{k+j})$ whenever $w_kw_j\ne0$. Since $w_1^2\ne0$, $\Span{w_1^2}= \mathbb{A}_c(\epsilon_n^{2})=\Span{w_2}$. Using the fact that the principal powers of $w_1$ are nonzero we conclude by indiction that
$$
\Span{w_1^m}= \mathbb{A}_c(\epsilon_n^{m})=\Span{w_m}, \qquad \forall m.
$$
Hence \ref{prow1} follows. Furthermore, scaling  $w_m$, one may assume without loss of generality that
\begin{equation}\label{w1m}
w_m=w_1^m, \qquad  \forall m,
\end{equation}
which proves \ref{prow3}. Finally, suppose that \ref{prow3} holds and suppose by contradiction that $w_kw_m=0$ for some $k,m$. Then $w_1^kw_1^m=0$ and $k\ne m$ by \ref{ax3}  in Proposition~\ref{pro:obser}. Suppose that $k<m$. By the cancelation property, $w_1^{k+1}w_1^m=(w_1w_1^k)w_1^m=0$, etc, implying in $m-k$ steps that $w_1^mw_1^m=0$, a contradiction with \ref{ax3} follows. This proves \ref{prow2}.
\end{proof}

\begin{remark}
Note that the subalgebra generated by any $w_j$ is \textit{not} power-associative. Indeed,
$$
w_1^2w_1^2=\frac{1}{\epsilon_n^2}(cw_1^2)(w_1w_1)=
\frac{1}{\epsilon_n^2}(cw_1)(w_1^2w_1)=
\frac{1}{\epsilon_n}w_1^4,
$$
hence $w_1^4\ne w_1^2w_1^2$.
\end{remark}

In fact,  one has the following general observation.

\begin{proposition}
\label{pro:weakly}
The subalgebra generated by $w_1$ is \textit{weakly power-associative} in the sense that any nonassociative monomial is proportional to the corresponding principal power:
\begin{equation}\label{powers1}
w_1^{\alpha}=bw_1^{\deg w_1^{\alpha}}, \qquad b=\epsilon_n^s, \quad 0\le s\le n-1.
\end{equation}
Furthermore, the product of two principal powers can be explicitly found by
\begin{equation}\label{powers}
w_1^kw_1^m=\frac{1}{\epsilon_n^{(k-1)(m-1)}}w_1^{k+m}, \qquad k,m\ge1.
\end{equation}
\end{proposition}

\begin{proof}
The proof of \eqref{powers}  follows from the recurrence identity
$$
w_1^kw_1^m=\frac{1}{\epsilon_n^{m}}(w_1w_1^{k-1})(cw_1^{m})=
\frac{1}{\epsilon_n^{m-1}}w_1(w_1^{k-1}w_1^{m})
$$
by induction.
Since the general nonassociative monomial is a successive product of principal power, one arrives at \eqref{powers1}.
\end{proof}

\begin{proof}[Proof of Theorem~\ref{the:uniq}]
Let $\mathbb{A}$ be a medial generic  isospectral  algebra. Given an idempotent $c$, let $\{w_1^i\}_{0\le i\le n-1}$ be a Peirce eigenbasis chosen accordingly \ref{prow3} in Proposition~\ref{pro:w}.
The algebra structure of an algebra is completely determined by product of its basis elements \eqref{powers}. This implies that two medial generic  isospectral  algebras of equal dimensions are isomotphic.
\end{proof}

\section{The cyclotomic polynomial model}\label{sec:polynom}
\subsection{The definition of $\mathscr{C}_n$}
Let $n\ge 2$. Consider the quotient
$$
\mathscr{C}_n=\Field[z]/z^n-1 %,\quad  \text{where $C_z=z^n-1$,}
$$
with a (commutative but nonassociative) multiplication on $\mathscr{C}_n$ defined by
$$
p(z)\circ q(z)=p(\epsilon_n z)q(\epsilon_n z) \mod z^n-1,
$$
where $\epsilon_n$ is a primitive root of unity of order $n$. Denote the corresponding algebra by $\mathscr{C}_n(z,\circ)$, or just $\mathscr{C}_n$ when there is no ambiguity. Such defined multiplication is obviously commutative.

\begin{remark}
The multiplication on $\mathscr{C}_n$ is an \textit{isotopy} of the usual (associative commutative) multiplication on $\mathbb{A}_n:=\Field[z]/z^n-1$. Indeed, define $f:\mathbb{A}_n\to \mathscr{C}_n$ by $f(p(z))=p(\epsilon_n z)$, then $h$ is a nonsingular linear transformation and
$$
p(z) \circ q(z)=h(p(z)q(z)).
$$
Thus, $\mathscr{C}_n$ is an isotopy of $\mathbb{A}_n$ by a cyclic linear substitution.
\end{remark}

\begin{example}
Let us demonstrate the multiplication by the table of multiplication of monomials for $n=3$:

\begin{center}\renewcommand\arraystretch{1.4}
\begin{tabular}{c|ccc}%\setlength\arraycolsep{1.4pt}
 & $1$ & $z$& $z^2$\\\hline
$1$& $1$ & $\epsilon_3 z$& $\epsilon_3^2 z^2$\\
$z$& $\epsilon_3 z$ & $\epsilon_3^2 z^2$& $1$\\
$z^2$& $\epsilon_3^2 z^2$ & $1$& $\epsilon_3 z$\\
\end{tabular}
\end{center}
This in particular shows that
$$
1\circ(1\circ z)=1\circ \epsilon_3 z=\epsilon_3^2 z,
$$
while
$$
(1\circ 1)\circ z=1\circ z=\epsilon_3 z.
$$
A similar property holds for any $n\ge 2$ which shows that the algebra $\mathscr{C}_n$ is nonassociative.
\end{example}

The element $1$ is obviously an idempotent of $\mathscr{C}_n$ for any $n$ and furthermore
$$
L^\circ_1z^k=\epsilon_n^kz^k,
$$
i.e. $z^k$ is an eigenvector of $L^\circ_1$ with eigenvalue $\epsilon_n^k$. In other words, $\{z^k\}_{0\le k\le n-1}$ is a Peirce eigenbasis  of $1$. Each eigenvalue is simple.

The main result of this section is

\begin{theorem}\label{the:model}
The algebra $\mathscr{C}_n$ is medial. Furthermore,
\begin{equation}\label{Ln}
L^{\circ n}_{p(z)}= \Delta(p)\mathbf{1},
\end{equation}
where
$$
\Delta(p):=p(1)p(\epsilon_n)p(\epsilon_n^2 )\cdots p(\epsilon_n^{n-1}):\mathscr{C}_n\to \Field
$$
is a multiplicative homomorphism.
\end{theorem}

\begin{proof}
Indeed, $p(z)\circ q(z)=p(\epsilon_n z)q(\epsilon_n z)$ and $r(z)\circ s(z)=s(\epsilon_n z)r(\epsilon_n z)$, therefore
$$
(p(z)\circ q(z))\circ (r(z)\circ s(z))=p(\epsilon_n^2 z)q(\epsilon_n^2 z)
r(\epsilon_n^2 z)s(\epsilon_n^2 z)\mod z^n-1.
$$
The latter product is totally symmetric in $p,q,r,s$ which implies the first claim. Next, note that
$$
L^{\circ2}_{p(z)}q(z)=p(z)\circ (p(z)\circ q(z))\equiv p(\epsilon_n z)p(\epsilon_n^2 z)
q(\epsilon_n^2 z)\mod z^n-1, \quad \text{etc}
$$
hence by induction
\begin{align*}
L^{\circ n}_{p(z)}q(z)&\equiv  p(\epsilon_n z)p(\epsilon_n^2 z)\cdots
p(\epsilon_n^{n} z)q(\epsilon_n^n z)\mod z^n-1\\
&\equiv  p(\epsilon_n z)p(\epsilon_n^2 z)\cdots
p(z)q(z)\mod z^n-1\\
&\equiv  P(z)q(z)\mod z^n-1,
\end{align*}
where
$$
P(z)=p(z)p(\epsilon_n z)p(\epsilon_n^2 z)\cdots p(\epsilon_n^{n-1} z).
$$
Since $P(\epsilon_n^i z)=P(z)$ for all $i$, there exists a polynomial $Q$ such that $P(z)=Q_p(z^n)$. Therefore
$$
P(z)=Q_p(z^n)\equiv Q_p(1)\equiv P(1)\mod z^n-1.
$$
This shows that
$$
L^{\circ n}_{p(z)}q(z)\equiv  P(1) q(z)\mod z^n-1,
$$
which proves \eqref{Ln}. Finally,
$$
\Delta(p(z)q(z))=\prod_{i=0}^{n-1}p(\epsilon_n^i)q(\epsilon_n^i)=
\prod_{i=0}^{n-1}p(\epsilon_n^i)\prod_{i=0}^{n-1}q(\epsilon_n^i)=\Delta(p(z))\Delta (q(z)),
$$
hence $\Delta$ is a multiplicative homomorphism.
\end{proof}

It follows from the Sylvester formula for the resultant of two polynomials that
\begin{equation}\label{result}
\Delta(p):=\prod_{z:z^n-1=0}p(z)=\mathcal{R}(p(z),z^n-1)=\det \mathrm{circ}(p),
\end{equation}
where $\mathrm{circ}(p)$ is the circulant matrix associated with $p$.

\subsection{The idempotents in $\mathscr{C}_n$}
We have seen above that the element $1$ is an idempotent of $\mathscr{C}_n$. Below we consider the general idempotents of $\mathscr{C}_n$.

Let $p(z)\in \mathscr{C}_n$. Define the  $\deg_R p(z)$ as the smallest degree of a polynomial $q$ such that $p\equiv q\mod z^n-1$. It is easy to see that there exists a unique polynomial of this kind.

Let $p(z)\in \mathscr{C}_n$ be an idempotent distinct from $1$. Then the idempotency relation reads as follows:
$$
p(z)=p(z)\circ p(z)=p^2(\epsilon_n z)\mod z^n-1,
$$
which is equivalent to the existence of a polynomial $q(z)$ such that
\begin{equation}\label{pq}
p^2(\epsilon_n z)-p(z)=Q(z)(z^n-1)
\end{equation}
Since $p(z)$ distinct from $1$, the polynomial $Q(z)$ is not identically zero.

\begin{proposition}
\label{pro:det1}
If $p(z)\in \Idm(\mathscr{C}_n)$ and $p(z)\not\equiv 1$ then $\Delta(p)=1$. Moreover, the Peirce spectrum of $p(z)$ is exactly
\begin{equation}\label{Peirceexact}
\sigma(p(z))=\{\epsilon_n^i:0\le i\le n-1\}.
\end{equation}
In particular, any idempotent $p(z)\in \mathscr{C}_n$ is semi-simple.
\end{proposition}

\begin{proof}
Let $p(z)$ be a nonzero idempotent in $\mathscr{C}_n$.
For any $0\le k\le n-1$, $\epsilon_n^k$ is a root of the right hand side of \eqref{pq}, therefore
\begin{equation}\label{pzetak}
p^2(\epsilon_n^{k+1})=p(\epsilon_n^k).
\end{equation}
Taking the product readily yields by virtue of the cyclicity of $\epsilon_n^k$ that
$$
\Delta(p)=\prod_{k=0}^{n-1}p(\epsilon_n^k)\in\{0,1\}.
$$
Let us show that $\Delta(p)=0$ is impossible. Indeed, if $\Delta(p)=0$ then $p(\epsilon_n^k)=0$ for some $k$. It follows from \eqref{pq} that $p^2(\epsilon^{k+1})=0$, thus $\epsilon^{k+1}$ also a root of $p$. This implies that all $\epsilon^i$, $0\le i\le n-1$ are the roots of $p$, i.e. $z^n-1$ divides $p(z)$, therefore implying that $p\equiv 0$ in $\mathscr{C}_n$. Therefore, if $p(z)$ is an idempotent then $\Delta(p)=1$. Next, by \eqref{Ln}
$L^{\circ n}_{p(z)}= \mathbf{1},$ which implies that any eigenvalue of $L^\circ_{p(z)}$ is a root of unity of degree $n$.

Let us show that  the Peirce spectrum is given by \eqref{Peirceexact}. If $p(z)\in \Idm(\mathscr{C}_n)$ then by the above $\Delta(p)=1$, thus \eqref{Ln} yields
\begin{equation}\label{Ln1}
L^{\circ n}_{p(z)}=\mathbf{1}.
\end{equation}
This implies that if $\lambda\in \sigma(p(z))$ then $\lambda^n-1=0$. Suppose that $\lambda=\epsilon_n^k$, $0\le k\le n-1$. Consider $q(z)=p(z)z^k$. Then using \eqref{pq}
$$
q(z)\circ p(z)\equiv q(\epsilon_n z)p(\epsilon_n z)\equiv
\epsilon_n^kz^kp(\epsilon_n z)p(\epsilon_n z)\equiv \epsilon_n^kz^kp(z)\equiv
\epsilon_n^kq(z)\mod z^n-1,
$$
which shows that $q(z)=z^kp(z)$ is a (nonzero in $\mathscr{C}_n$) eigenvector with eigenvalue $\epsilon_n^k$. For  dimensional reasons, the equality \eqref{Peirceexact} holds and each eigenvalue is simple. Finally, note that since the eigenvectors $w_k=z^kp(z)$, $k=0,1,\ldots, n-1$, correspond to different eigenvalues they form a basis in $\mathscr{C}_n$. Thus $p(z)$ is a semi-simple idempotent.
\end{proof}

\begin{proposition}\label{pro:Rgeneric}
The algebra $\mathscr{C}_n$ is generic.
\end{proposition}

\begin{proof}
First note that $\Idm(\mathscr{C}_n)$ is nonempty because, for example,  $p_0(z)=1\in \Idm(\mathscr{C}_n)$. We  claim that there no nontrivial 2-nilpotents in $\mathscr{C}_n$. Indeed, suppose that  $p(z)\circ p(z)=0$ in $\mathscr{C}_n$. The latter yields $p^2(z\epsilon_n)=(z^n-1)Q(z)$ for some polynomial $Q(z)$, or after change of variables $\zeta=\epsilon_n z$ we have $p^2(\zeta)=(\zeta^n-1)Q_1(\zeta)$. The latter implies that $p(\zeta)$ has $n$ distinct roots $\epsilon_n^k$, $0\le k\le n-1$, thus, by virtue of $\deg p\le n-1$, we conclude that  $p(\zeta)\equiv 0$. This proves our claim. Finally, by Proposition~\ref{pro:det1} $\half12\not\in\sigma(p(z))$ for any idempotent $p(z)$ of $\mathscr{C}_n$, and therefore $\half12\not\in\sigma(\mathscr{C}_n)$. Applying Theorem~\ref{th:gener} we conclude that  $\mathscr{C}_n$ is a generic algebra.
\end{proof}

\begin{proposition}\label{pro:Risospectral}
The algebra $\mathscr{C}_n$ is isospectral.
\end{proposition}

\begin{proof}
By Theorem~\ref{the:model} and Proposition~\ref{pro:Rgeneric}, $\mathscr{C}_n$ is a medial generic  algebra and since $0\not\in \sigma(p(z))$ for any $p(z)\in \Idm(\mathscr{C}_n)$, we also conclude that $\mathscr{C}_n$ is reduced. By Corollary~\ref{cor:isosp}, $\mathscr{C}_n$ is isospectral. Alternatively, since $\mathscr{C}_n$ is generic it has the maximal number of idempotents, and by Proposition~\ref{pro:det1} each idempotent is semi-simple with the same spectrum.
\end{proof}

\section{The quasigroup structure of isospectral medial alegrbas}\label{sec:structure}

As we already seen above, idempotents of a medial algebra always form a multiplicative magma. In fact, by the medial magma property one can deduce that the set of idempotents is a quasigroup, see the definitions in section~\ref{sec:prelim} above.

\begin{proposition}\label{pro:again}
Let $\mathbb{A}$ be a medial generic  isospectral algebra of dimension $n$. Then $\Idm(\mathbb{A})$ is a idempotent medial commutative  quasigroup of order $2^n-1$.
\end{proposition}

\begin{proof}
By \ref{ax2} in Proposition~\ref{pro:idem}, the product of two idempotents is a nonzero idempotent. Let $c\in \Idm(\mathbb{A})$ be an arbitrary idempotent. Then by the above, $L_c:\Idm(\mathbb{A})\to \Idm(\mathbb{A})$ and $L_c$ is a bijection because if $cc_1=cc_2$ for some $c_1\ne c_2$ then $c_1-c_2$ is a nonzero eigenvector of $L_c$ with eigenvalue $0$ which contradicts to \eqref{eqiso}. This proves that $\Idm(\mathbb{A})$ is a quasigroup. It is obviously idempotent medial and commutative.
\end{proof}

In fact we have a stronger conclusion

\begin{proposition}\label{pro:abelian}
Let $\mathbb{A}$ be a medial generic  isospectral algebra of dimension $n$  and $c\in\Idm(\mathbb{A})$. Then $Q:=(\Idm(\mathbb{A}), \boxplus)$ is an abelian group with the identity $c$ with respect to the new binary operation $x\boxplus y:=L_c^{-1}(xy)$.
\end{proposition}

\begin{proof}
By Proposition~\ref{pro:ass}, $x\circ y=L_c^{-1}(xy)$ is an isotopy of the algebra $\mathbb{A}$ such that $(\mathbb{A},\circ)$ becomes a unital commutative associative algebra with unit $c$. Note that by the definition  $x\boxplus y=x\circ y$ on the subset of nonzero idempotents of $\mathbb{A}$ and by Proposition~\ref{pro:again}, $\boxplus$ preserves $\Idm(\mathbb{A})$, which implies the desired conclusion.
\end{proof}

Proposition~\ref{pro:abelian} shows the existence of a `hidden' abelian group structure on the set of idempotents of a medial algebra. This fact also follows from Proposition~4.1 in the paper of A.~Leibak and P~Puusemp \cite{Leibak2014}.

But, in fact, we have a much stronger conclusion.

\begin{theorem}\label{the:Zn}
The abelian group $Q:=(\Idm(\mathbb{A}), \boxplus)$  in Proposition~\ref{pro:abelian} is a cyclic group of order $2^n-1$.
\end{theorem}

\begin{proof}
Given a fixed $c\in \Idm(\mathbb{A})$, we denote by $T=L_c^{-1}:\mathbb{A}\to \mathbb{A}$ an algebra isomorphism. By Proposition~\ref{pro:abelian},
$Q:=(\Idm(\mathbb{A}), \boxplus)$ is an abelian group with respect to $x\boxplus y:=T(xy)$ and with unity element $c$ of order $2^n-1$. Since $T(xy)=T(x)T(y)$ on the algebra level and $T$ is invertible, we also have
\begin{equation}\label{product}
xy=T^{-1}(x)\boxplus T^{-1}(y)=T^{-1}(x\boxplus y).
\end{equation}
In order to show that $Q$ is in fact \textit{cyclic}, i.e. $Q\cong \mathbb{Z}_{2^n-1}$, it suffices to show that \textit{there does not exist a positive integer $1\le m<2^n-1$ such that}
\begin{equation}\label{there}
\underbrace{x\boxplus\ldots\boxplus x}_{\text{$m$ times}}=c \qquad \text{holds for any element $x\in Q$}.
\end{equation}
We argue by contradiction and assume that there exist a positive integer $1\le m<2^n-1$ satisfying \eqref{there}. Let us write $m=2^{m_1}+\ldots +2^{m_{k-1}}+2^{m_k}$, $n>m_k>\ldots >m_1\ge0$, the binary decomposition of $m$. Using the doubling operator $T$ we obtain
$2\otimes x=x\boxplus x=T(xx)=T(x)$ for any $x\in Q$, hence $2^s \otimes x=T^{s}(x)$ implying by virtue of \eqref{there}
\begin{equation}\label{productm}
\underbrace{x\boxplus\ldots\boxplus x}_{\text{$m$ times}}=T^{m_k}(x)\boxplus \ldots \boxplus T^{m_1}(x)=c.
\end{equation}
Combining  the terms into a nonassociative products we obtain using \eqref{product} and then the fact that $T$ is an algebra isomorphism that
$$
T^{m_{2}}(x) \boxplus T^{m_1}(x)=T(T^{m_{2}}(x) T^{m_1}(x))=T^{m_{2}+1}(x) T^{m_1+1}(x).
$$
Arguing similarly, we get
$$
T^{m_{3}}(x) \boxplus T^{m_{2}}(x) \boxplus T^{m_1}(x)=T^{m_{3}+1}(x)\bigl(T^{m_{2}+2}(x) T^{m_1+2}(x)\bigr)
 \quad \text{etc}
$$
Define
$$
H(x):=T^{m_{k}+1}(x)\bigl(T^{m_{k-1}+2}(x)(\ldots(T^{m_{2}+k-1}(x) T^{m_1+k-1}(x)\bigr).
$$
By the linearity of $T=L_c^{-1}$ on $\mathbb{A}$, $H(x)$ is a nonassociative polynomial of degree $k\le n-1$. Also $H(0)=0$ and by \eqref{productm},  $H(x)=c$ for all $x\in \Idm(\mathbb{A})$. Hence by the syzygy relation \eqref{syzy1} we obtain
$$
(2^n-1)c=\sum_{x\in Q}H(x)=0,
$$
a contradiction finishes the proof.

\end{proof}

\begin{corollary}[Theorem~\ref{the:Znn}]
The quasigroup $\Idm(\mathbb{A})$ (with respect to the original multiplication in $A$) is strongly isotopic to the quasigroup $(\mathbb{Z}_{2^n-1},\bullet)$ with multiplication defined by
\begin{equation}\label{IJ}
u\bullet v\equiv \frac{1}{2}(u+v)\mod (2^n-1),\qquad \forall u,v\in \mathbb{Z}_{2^n-1},
\end{equation}
\end{corollary}

\begin{proof}
Given a fixed idempotent $c\in \Idm(\mathbb{A})$ let $Q=(\Idm(\mathbb{A}), \boxplus)$ be the cyclic group of order $2^n-1$ with unit $c$ and defined explicitly by $x\boxplus y=T(xy)$, where $T=L_c^{-1}$. By Theorem~\ref{the:Zn} there exists an isomorphism  $\phi$ of $Q$ to the  cyclic group $(\mathbb{Z}_{2^n-1}, +)$ of integers with modular addition. By the definition, 
\begin{equation}\label{ppree}
\phi(T(xy))=\phi(x\boxplus y)=\phi(x)+\phi(y)\qquad \text{for any $x,y\in Q$.}
\end{equation}
On the other hand, since $uu=u$ for any $u\in \Idm(\mathbb{A})$ then $T(u)=T(uu)=u\boxplus u$, hence $\phi(T(u))=2\phi(u)$, and combining this with \eqref{ppree} we obtain 
$$
2\phi(xy)=\phi(x)+\phi(y),
$$
i.e. using the notation in \eqref{IJ},  
$$
\phi(xy)=\frac12(\phi(x)+\phi(y))=\phi(x)\bullet \phi(y),
$$
which shows that $\phi:\Idm(\mathbb{A})\to (\mathbb{Z}_{2^n-1},\bullet)$ is a strong isotopy, hence implying the desired conclusion. 
\end{proof}

\section{Discussion and final remarks}
\subsection{Medial algebra extensions}
Let $(Q,\circ)$ be an idempotent medial commutative quasigroup, IMC-quasigroup, for short, following \cite{LeibakPuusemp}. Note that the order $|Q|$ of an IMC-quasigroup is always an \textit{odd} number, see, for example, Corollary~2.5 in \cite{Kinyon09}.

\begin{definition}
An algebra $\mathbb{A}$ over a field $\Field$ is called an \textit{medial extension} of $Q$ if there exists a multiplicative monomorphism (injective homomorphism)  $f:Q\to \Idm(\mathbb{A})$ and $f(Q)$ generates $\mathbb{A}$.
\end{definition}

We see from Proposition~\ref{pro:suff} that any medial extension of a quasigroup is  a medial algebra.

It follows from the definition that
$\dim A\le |Q|\le \mathrm{card}( \Idm(\mathbb{A}))$. Furthermore, if algebra $\mathbb{A}$ is generic then one also has
$$
\dim A\le |Q|\le \mathrm{card}( \Idm(\mathbb{A}))\le 2^{\dim A}-1.
$$
To distinguish the extremal cases in the latter inequality, we say that a medial extension is \textit{minimal} (resp. \textit{maximal}) if $|Q|=  2^{\dim A}-1$ (resp. if $|Q|=\dim \mathbb{A}$).

A natural question in the present context arises: does an arbitrary IMC-quasigroup obey a minimal medial extension?

The following proposition shows that any IMC-quasigroup obeys a maximal medial extension.

\begin{proposition}
\label{pro:extens}
Let $Q$ be an IMC-quasigroup and let $\mathbb{A}$ be the free module generated by $Q$ as a basis. Define a commutative nonassociative algebra structure on $\mathbb{A}$ by the natural extension of the quasigroup multiplication on $Q$. Then $\mathbb{A}$ is a maximal medial extension of $Q$.
\end{proposition}

\begin{proof}
Follows immediately from the definition and Proposition~\ref{pro:suff}.
\end{proof}

%Thus, by Proposition~\ref{cor:exists}, any IMC-quasigroup $(Q,\circ)$ obeys at least one medial extension. In many cases, there exists  a medial algebra generated by a proper subset  $S \subsetneq Q$ and an algebraic closed field $\Field$.

\subsection{Construction of a medial algebra from a medial quasigroup}

Using Proposition~\ref{pro:suff}, we are ready to produce some simple explicit of medial algebras. As we already mentioned above, the order of any IMC-quasigroup is an odd number. An example below shows that for any odd number there exists at least one IMC-quasigroup.
%To this note that if a basis of an algebra $\mathbb{A}$ satisfies the conditions of Proposition~\ref{pro:suff} then its elements are idempotents. Indeed,

Let $N$ be an odd positive integer and let $\mathbb{Z}_N=\mathbb{Z}/N\mathbb{Z}$  denote the commutative ring of integers modulo $N$. We denote by $\bar x$ the residue class containing $x\in \mathbb{Z}$, or just by $x$ if this causes no ambiguity.

\begin{proposition}
\label{pro:exist}
 Then $\ZN:=(\mathbb{Z}_N,\circ)$ with multiplication defined by
\begin{equation}\label{ij}
\bar{x}\circ \bar{y}\equiv \frac{1}{2}(\bar{x}+\bar{y})\mod N,\qquad \forall \bar{x},\bar{y}\in \mathbb{Z}_N,
\end{equation}
is an IMC-quasigroup of order $N$.

\end{proposition}

\begin{proof}
Since $2$ is invertible in $\mathbb{Z}_N$, the multiplication $\circ$ is well defined and commutative. The medial magma identity follows immediately from
$$
(\bar{x}\circ\bar{y})\circ(\bar{z}\circ \bar{w})=
\frac{1}{4}(\bar{x}+\bar{y}+\bar{z}+\bar{w})
$$
whereas the left hand side is completely symmetric in all four variables. Next, to verify the quasigroup property let $\bar{x},\bar{y}\in \mathbb{Z}_N$. Then $\bar{w}=2\bar{y}-\bar{x}$ satisfies $\bar{w}\circ \bar{x}=\bar{y}$. If $\bar{w}_1\circ \bar{x}=\bar{y}$  then $\frac12(\bar{w}_1-\bar{w})= 0$, hence $\bar{w}=\bar{w}_1$. This shows that $\circ$ satisfies the Latin square property, as desired.  Finally, $\bar{x}\circ \bar{x}\equiv \bar{x}$, thus all elements of $\ZN$ are idempotents. This shows that $\ZN$ is an IMC-quasigroup. Clearly, $|\ZN|=N$.
\end{proof}

\begin{remark}
Comparing the above multiplication \eqref{ij} with the Bruck-Murdoch-Toyoda theorem (Theorem~\ref{th:murd}), one can identify the corresponding abelian group  $(\mathbb{Z}_N,+)$ with $g=0$ and automorphisms $\phi=\psi$ defined as multiplication by $\frac12$. The latter choice is essentially the only possible one as a straightforward verification by virtue of the Bruck-Murdoch-Toyoda theorem shows; see also Proposition~4.1 in \cite{Leibak2014} for a general abelian group. See some further discussion in \cite{Leibak2014}, \cite{LeibakPuusemp}.
\end{remark}

To proceed  we need to characterize the $\circ$ multiplication in more detail. To this end, recall that given two relative prime $a$ and $N$, the smallest integer $n$ satisfying $a^n\equiv 1 \mod N$ is called the multiplicative order of $a$ modulo $N$, denoted by $\ord_N(a)$. It is well known in the classical number theory that $\ord_N(a)$ divide $\varphi(N)$, where  $\varphi$ is  Euler's totient function. As  a consequence of Lagrange's theorem, $\ord_N(a)$ always divides $\varphi(N)$.

Let ${x},{y}\in \ZN$. Define the \textit{order} $\ord^\circ_x y$ of an element $y$ with respect to $x$ as the minimal positive integer $p$ such that $L_{{x}}^p{y}={y}$.

\begin{proposition}
\label{pro:order}
The order is well defined and given by $\ord^\circ_x y=\omega_N(x-y)$, where
\begin{equation}\label{phi2}
\omega_N(m):=\min\{p\in \mathbb{N}: (2^p-1)m\equiv 0\mod N\}.
\end{equation}
In particular, the order depends only on (the residue class of) $|x-y|$ and
\begin{equation}\label{phi3}
\ord^\circ_x y=\ord^\circ_y x.
\end{equation}
 \end{proposition}

\begin{proof}
First note that $\ord^\circ_x y$ is well defined. Indeed, since $\ZN$ is finite, then by the pigeonhole principle there exists integers $q>k\ge 0$ such that $L_x^qy=L_x^ky$. By the quasigroup cancelation property we have $L_x^{q-k}y=y$. Define $p$ as the minimal positive integer such that  $L_x^py=y$ holds. Since
$$
L_x^iy\equiv \frac{1}{2^i}y+(\frac{1}{2^i}+\ldots+\frac{1}{2})x\equiv
\frac{1}{2^i}y+(1-\frac{1}{2^i})x\mod N,
$$
we have  $\frac{2^p-1}{2^p}(x-y)\equiv 0\mod N$, or, since $N$ is odd,
\begin{equation}\label{phi1}
(2^p-1)(x-y)\equiv 0\mod N.
\end{equation}
This implies our claim by virtue of \eqref{phi2}.
Then \eqref{phi3} follows from \eqref{phi2}.
\end{proof}

Since $2$ and $N$ are relatively prime, by Euler's theorem $2^{\varphi(N)}-1\equiv 0\mod N$. Thus, $p\le \varphi(N)$.
In fact, the above argument also implies that $\ord^\circ_x y\le \ord_N (a)$, the multiplicative order of $a$ modulo $N$.

\begin{corollary}
If $N$ is an odd prime then $\ord^\circ_x y=\ord_N(2)$ unless $x=y$ in $\mathbb{Z}_N$.
\end{corollary}

\begin{proof}
Indeed, for all  distinct $x$ and $y$ in $\mathbb{Z}_N$,  $x-y$ is coprime with $N$, therefore \eqref{phi2} is equivalent to
$$
\ord^\circ_x y=\min\{p\in \mathbb{N}: 2^p\equiv 1\mod N\}=\ord_N(2),
$$
hence the conclusion follows.
\end{proof}

For composite odd $N$, the situation is more complicated. Since we are primarily interested in the minimal medial extensions, the most relevant for us case is when $N=2^n-1$, $n=2,3,\ldots$ We consider this case in the next section.

\subsection{The particular case $N=2^n-1$}
We consider below the case $N=2^n-1$, where $n\ge 2$ is an integer. Then it follows from \eqref{phi2} that
$$
\ord^\circ_x y\le n.
$$

\begin{proposition}
\label{pro:euler}
The set $P(n)$ of all possible orders $p=\ord^\circ_x y$, where $x\ne y$, is exactly the set of all integers less or equal $n$ and having a nontrivial common divisor with $n$.
\end{proposition}

\begin{proof}
First note that if $k,m\ge2$ then
$$
2^{km}-1=(2^k-1)(2^{k(m-1)}+\ldots+1)
$$
hence $2^k-1$ divides $2^{km}-1$.

Now let  $s:=\gcd(p,n)>1$ and  $p<n$. Then by the above, $r:=2^s-1$ is a nontrivial common divisor of $2^p-1$ and $2^n-1$, therefore $p\in P(n)$. In the converse direction, let $p\in P(n)$. Then it follows from \eqref{phi2} that $2^p-1$ and $2^n-1$ have a nontrivial common divisor,  say $r$. Then $r$ is odd and divides $2^n-2^p$, thus also divides $2^{n-p}-1$. Arguing as in the proof of the Euclidean algorithm, we conclude that $r$ also divides $2^s-1$, where  $s:=\gcd(p,n)$. Since $r$ is nontrivial, we have $s\ge2$, thus $p$ and $n$ also have a nontrivial divisor. The proposition follows.
\end{proof}

\begin{example}
We illustrate the above by a composite number $N=15=2^4-1$, where $\ord_{15}2=4$. Then $\frac{1}{2}\equiv 8 \mod 15$, thus
$$
x\circ y\equiv  8(x+y)\mod 15.
$$
Let $x=1$ and $y=2$. Then the sequence $L_x^ky$ is given explicitly for $k=0,1,2,3,\ldots$ by
$$
2,\,9,\,5, \,3, 2\, \ldots\quad \Rightarrow\quad \ord^\circ_12=4,
$$
in contrast to $x=1$ and $y=6$, where the corresponding $L_x^ky$ are
$$
6, \,11, 6, \ldots\quad \Rightarrow\quad \ord^\circ_16=2.
$$
By Proposition~\ref{pro:euler}, $P(4)=\{2,4\}$. Indeed, for a general $x, y\in \mathbb{Z}_{15}$ we have
$$
\ord^\circ_x y=
\left\{
\begin{array}{cl}
1& \text{if $x-y=0$}\\
2& \text{if $(x-y,5)=5$}\\
4& \text{if $(x-y,5)=1$}
\end{array}
\right.
$$
The action of $L_1$ on $\mathbb{Z}^\circ_{15}$ gives the following orbits:
$$
\begin{array}{ccccccc}
1&&&\\
6&\to&  11&&\\
2& \to& 9 &\to& 5  &\to& 3 \\
4& \to &10 &\to &13 &\to& 7\\
8& \to &12 &\to &14 &\to& 15\\
\end{array}
$$

\subsection{An alternative model of $\mathbb{A}_n$}
It would be very desirable to have some more explicit information about the idempotents of $\mathscr{C}_n$ , see section~\ref{sec:polynom}. Unfortunately, the  idempotency relation \eqref{pq} is  not very illuminating and even in the lower-dimensional case $n=3$ leads to a rather difficult algebraic manipulations. We notice that the identities \eqref{pzetak} uniquely determine the set of all idempotents of  $\mathscr{C}_n$. Indeed, let us consider the system
\begin{equation}\label{system}
\left\{
\begin{array}{ll}
p^2(\epsilon_n^1)&=p(1)\\
p^2(\epsilon_n^2)&=p(\epsilon_n)\\
%p^2(\epsilon_n^{3})&=p(\epsilon_n^2)\\
\ldots&=\ldots\\
p^2(\epsilon_n^{n})&=p(\epsilon_n^{n-1})
\end{array}
\right.
\end{equation}
Then the above system is a system of $n$ quadratic equations in variables $(x_0,x_1,\ldots,x_{n-1})$. By B\'ezout's theorem, \eqref{system} has either infinitely many or less than $2^n$ complex solutions (counting the trivial zero solution). It is not difficult to show that it has exactly $2^n$ distinct solutions using the fact that the Jacobian matrix is nonsingular at any solution (cf. (9) in \cite{KrTk18a} which is essentially equivalent to that $\frac12\not\in \sigma(\mathscr{C}_n)$). Due to the symmetry of \eqref{system}, it can be reduced to an algebraical equation of degree $2^n-1$. A more careful analysis with help of Galois theory reveals that any idempotent $p$ is a polynomial with coefficients in the cyclotomic extension $\mathbb{Q}[\epsilon_{2^n-1},\epsilon_n]$

\subsection{Iterations}
Note also that since $\mathscr{C}_n$ is medial algebra it makes it easy to construct idempotents if one knows some of them by combining the following steps:

\begin{enumerate}[label=(\Alph*)]
  \item \label{idem1}
  If $p(z)$ and $q(z)$ are idempotents then $p(z\epsilon_n^k)q(z\epsilon_n^k)$ is an idempotent too.
  \item \label{idem2}
  Applying the above to $q=1$, $p(z\epsilon_n^k)=L_{1}^{\circ k}p(z)$ is an idempotent in $\mathscr{C}_n$.
\end{enumerate}

Let $p(z)$ be an idempotent in $\mathscr{C}_n$ distinct from $1$ such that let $\deg_R p=\deg p(z)$. By the medial property Proposition~\ref{pro:idem}, $1\circ p$ is an idempotent too. This shows that the iterations
$$
p(z),\, p(\epsilon_n z),\, \ldots,\, p(\epsilon_n^k z),\, \ldots,\, p(\epsilon_n^{n-1}z)
$$
are also idempotents. Since $\deg p\le n-1<n$, $p(\epsilon_n^k z)\ne p(\epsilon_n^m z)$ as elements in $\mathscr{C}_n$ unless $p(\epsilon_n^k z)= p(\epsilon_n^m z)$ as polynomials. Suppose that $p(\epsilon^k z)= p(\epsilon^m p(\epsilon^m z)$ for some $0\le k<m\le n-1$. Then $p(\epsilon^{k-m} z)= p(z)$. The latter is possible if and only if
$$
p(z)=p_1(z^\nu), \qquad \text{where $\nu=\frac{n}{\gcd(n,k-m)}$.}
$$

%The standard argument shows that in fact the distinct idempotents are exactly
%$$
%p(z),\, p(\epsilon_n z),\, \ldots,\, p(\epsilon_n^k z),\, \ldots,\, p(\epsilon_n^{n-1}z)
%$$

%\begin{proposition}
%For any $0\le m\le n-1$ there exists exactly $2^m$ idempotents of degree
%\end{proposition}

\end{example}

\subsection{Appendix: Non-generic isospectral algebras}
Finally we give an example of a three-dimensional non-generic isospectral algebra which is \textit{non}-Hsiang. This is a part of a more general construction, we consider it in full generality elsewhere.

Let us define  a three dimensional algebra $\mathbb{T}=\Field^3$, where $\Field$ is any field of characteristic not 2 and 3, with multiplication
\begin{equation}\label{app1}
x\circ y=(x_1y_1-\frac12x_2y_3-\frac12x_3y_2,\,x_2y_2-\frac12x_1y_3-\frac12x_3y_1,\,
x_3y_3-\frac12x_1y_2-\frac12x_1y_2 ),
\end{equation}
where $x=(x_1,x_2,x_3)$ and $y=(y_1,y_2,y_3)$ in a fixed basis of $\Field^3$.

%As a corollary of \eqref{adm}, one can show (Theorem~1 in \cite{ElOkubo}) that
%$$
%N((x\circ x))=N(x)^2.
%$$

\begin{remark}
It is easy to see that any element of $\mathbb{T}$ satisfies the so-called `adjoint identity'
\begin{equation}\label{adm}
(x\circ x)\circ (x\circ x)=N(x)x,
\end{equation}
with cubic form $N(x)$ given explicitly by
$$
N(x)=x_1^3 + x_2^3 + x_3^3-3x_1x_2x_3=(x_1 + x_2 + x_3) (x_1^2 - x_1 x_2 - x_1 x_3 + x_2^2 - x_2 x_3 + x_3^2).
$$
In the general case, an algebra satisfying \eqref{adm} is called in \cite{ElOkubo} an \textit{admissible cubic algebra}. We refer the interested reader to a recent paper of Elduque and Okubo\cite{ElOkubo} for more details and classification.
Any admissible algebra carries an invariant symmetric bilinear form $b$, i.e. it is metrized. In our case, the invariant bilinear form is just the Euclidean scalar product, $\scal{x}{y}=x_1y_1+x_2y_2+x_3y_3$. One can examine this property directly, alternatively we refer to the general formula (5) in \cite{ElOkubo}. We do not exploit the latter fact below but it explains why the Euclidean scalar product is relevant for $\mathbb{T}$.
\end{remark}

\begin{proposition}
The algebra $\mathbb{T}$ is isospectral. If $\Field$ is infinite, the algebra $\mathbb{T}$ is non-generic. The set of nonzero idempotents of $\mathbb{T}$ is the circle given by
\begin{equation}\label{ST}
\Idm(\mathbb{T})=\{(x_1,x_2,x_3):\, \sum_{i=1}^3x_i^2=\sum_{i=1}^3 x_i=1\}.
\end{equation}
Any idempotent $x\in \Idm(\mathbb{T})$ has spectrum $\{1,-\frac12,\frac12\}$.
\end{proposition}

\begin{proof}
Let us first prove that the set of nonzero idempotents is exactly $S:=\{(x_1,x_2,x_3):\, \sum_{i=1}^3x_i^2=\sum_{i=1}^3 x_i=1\}$. To see that $S\subset \Idm(\mathbb{T})$, we note that for any point $x\in S$, there holds
\begin{align*}
x_1^2-x_2x_3&=(1-x_2-x_3)^2-x_2x_3=1+x_2x_3+x_2^2+x_3^2-2(x_2+x_3)\\
&=x_2x_3-x_1^2+2x_1,
\end{align*}
implying $x_1^2-x_2x_3=x_1$, thus by symmetry implying $x\circ x=x$. Conversely, if the latter idempotent identity holds true, then $x_i^2-x_jx_k=x_i$ for any permutation $\{i,j,k\}=\{1,2,3\}$. Summing up we obtain
\begin{equation}\label{vw}
w-\frac12 (v^2-w)=v, \qquad w:=\sum_{i=1}^3x_i^2, \quad v:=\sum_{i=1}^3x_i.
\end{equation}
Next, subtracting the idempotents relations, one obtains
$$
(x_i-x_j)(x_1+x_2+x_3-1)=0, \qquad \forall 1\le i<j\le 3.
$$
If $x_1+x_2+x_3=1$, we have $v=1$ and by \eqref{vw} also $w=1$, hence $x\in S$. If $x_1+x_2+x_3\ne1$ then all $x_i$ are equal which immediately implies  $x=0$. This shows that $\Idm(\mathbb{T})\subset S$ and finishes the proof of \eqref{ST}. In particular, the obtained algebra is non-generic.

Next, set $e=(1,1,1)$. Then for any idempotent $x\in \Idm(\mathbb{T})$ we have
$$
e\circ x=(x_1-\frac12x_2-\frac12x_3,x_2-\frac12x_1-\frac12x_3,x_3-\frac12x_1-\frac12x_2)=\frac32 x-\frac12 e,
$$
hence
\begin{equation}\label{eeq}
(e-x)\circ x=\frac32 x-\frac12 e-x=-\frac12(e-x),
\end{equation}
i.e. $(e-x)$ is an eigenvector of $L^\circ_x$ with eigenvalue $-\frac12$. Similarly, define $u=e\times x$, where $\times$ is the (formal) cross-product in $\Field ^3$, i.e.
$u=(x_2-x_3, x_3-x_1,x_1-x_2).$ Then
$$
x\circ u=\frac{x_1+x_2+x_3}2 u=\frac12 u,
$$
hence $u$ is an eigenvector of $L^\circ_x$ with eigenvalue $\frac12$. Also, $x$ itself is an eigenvector of $L^\circ_c$ with eigenvalue $1$. In summary, we have an eigenbasis of $L^\circ_x$ consisting of $x, e-x, e\times x$. This implies that the  spectrum of $L^\circ_x$ is $\{1,-\frac12,\frac12\}$. Thus, all idempotents of $\mathbb{T}$ have the same spectrum, so the algebra is isospectral.  Finally, since the trace $\trace L^\circ_x=1\ne 0$ the algebra is not a Hsiang algebra.
\end{proof}

%\medskip
%\subsection*{Acknowledgment}
%The second author has been paritally supported by Stiftelsen GS Magnusons fond, grant MG2018-0042.

\bibliographystyle{plain}      % mathematics and physical sciences

%\bibliography{grant}

\def\cprime{$'$}

\end{document}